\documentclass[11pt]{amsart}
\usepackage{amsmath,amssymb,amsfonts,amsthm,amsopn}
\usepackage{latexsym,graphicx}
\usepackage{pb-diagram}
\newcommand{\tfa}{time-frequency analysis}

\newcommand{\stft}{short-time Fourier transform}

\newcommand{\modsp}{modulation space}

\newtheorem{theorem}{Theorem}[section]
\newtheorem{lemma}[theorem]{Lemma}

\newtheorem{proposition}[theorem]{Proposition}
\newtheorem{definition}[theorem]{Definition}

\newtheorem{example}[theorem]{Example}
\newtheorem{remark}[theorem]{Remark}

\newcommand{\beqa}{\begin{eqnarray*}}
\newcommand{\eeqa}{\end{eqnarray*}}

\newcommand{\field}[1]{\mathbb{#1}}
\newcommand{\bR}{\field{R}}        %  real numbers
\newcommand{\bN}{\field{N}}        %  natural numbers
\newcommand{\bZ}{\field{Z}}        %  whole numbers
\newcommand{\bC}{\field{C}}        %  complex numbers
\def\G{\mathcal{G}}

\def\la{\lambda}

\def\eps{\epsilon}

\def\cF{\mathcal{F}}              % Calligraphic Letters
\def\cS{\mathcal{S}}

\def\cG{\mathcal{G}}
\def\cM{\mathcal{M}}

\def\cC{\mathcal{C}}

\def\a{\aleph}

\def\rd{\bR^d}

\def\rdd{{\bR^{2d}}}
\def\zdd{{\bZ^{2d}}}

\def\lrd{L^2(\rd)}

\def\intrd{\int_{\rd}}

\def\R{\right)}

\def\<{\left<}
\def\>{\right>}

\def\mv1{M_v^1}

\def\phas{(x,\o )}
\def\mn{(m,n)}
\def\mn'{(m',n')}

\newcommand{\abs}[1]{\lvert#1\rvert}

\def\o{\xi}
\def\a{\alpha}
\def\b{\beta}

\def\R{\mathbb{R}}
\def\Ren{\mathbb{R}^d}
\def\Renn{\mathbb{R}^{2d}}

\def\sch{\mathcal{S}}

\def\Fur{\mathcal{F}}

\def\Sn2{S_{2}(L^{2}(\Ren))}
\def\S1{S_{1}(L^{2}(\Ren))}
\def\sig00{\sigma_{0,0}}
\def\la{\langle}
\def\ra{\rangle}

\begin{document}
\begin{abstract}
We perform a  time-frequency analysis of Fourier multipliers and, more generally, pseudodifferential operators with symbols of Gevrey, analytic and ultra-analytic regularity. As an application we show that Gabor frames, which provide optimally sparse decompositions for Schr\"odinger-type propagators, reveal to be an even more efficient tool for representing solutions to a wide class of evolution operators with constant coefficients, including weakly hyperbolic and parabolic-type operators. Besides the class of operators, the main novelty of the paper is the proof of {\it super-exponential} (as opposite to super-polynomial) off-diagonal decay for the Gabor matrix representation.
\end{abstract}

\title{Gabor representations of evolution operators}

\author{Elena Cordero,  Fabio Nicola  and Luigi Rodino}
\address{Dipartimento di Matematica,
Universit\`a di Torino, via Carlo Alberto 10, 10123 Torino, Italy}
\address{Dipartimento di Scienze Matematiche,
Politecnico di Torino, corso Duca degli Abruzzi 24, 10129 Torino,
Italy}
\address{Dipartimento di Matematica,
Universit\`a di Torino, via Carlo Alberto 10, 10123 Torino, Italy}

\email{elena.cordero@unito.it}
\email{fabio.nicola@polito.it}
\email{luigi.rodino@unito.it}
%\thanks{}

%\date{}
\maketitle
\section{Introduction}
A Gabor system is obtained by fixing a function $g\in L^2(\rd)$ and considering the time-frequency shifts
\begin{equation}\label{intro1}
\pi(\lambda)g=e^{2\pi i n x}g(x-m),\quad \lambda=(m,n)\in\Lambda,
\end{equation}
for some lattice $\Lambda\subset\rdd$. A Gabor system is a Gabor frame if there exist $A,B>0$ such that
\begin{equation}\label{intro2}
A\|f\|^2_{L^2}\leq \sum_{\lambda\in\Lambda}|\langle f,\pi(\lambda)g\rangle|^2\leq B\|f\|^2_{L^2}
\end{equation}
for every $f\in L^2(\rd)$, see for example \cite{ibero13,ibero30,ibero45}. With respect to frames of wavelets, curvelets and shearlets, in \eqref{intro1} dilations are then replaced by modulations. From the analytic point of view, this provides symmetric properties with respect to the Fourier transform and a simple treatment, cf.\  \cite{grochenig}.\par
Gabor frames turned out to be appropriate tools for many problems in time-frequency analysis, with relevant applications to signal processing and related issues in Numerical Analysis, see for example \cite{cordero-feichtinger-luef,str06}, and references there. More recently, attention has been addressed to the analysis of partial differential equations, especially the Schr\"odinger, wave and Klein-Gordon equations with constant coefficients \cite{bertinoro2,bertinoro3,bertinoro12,bertinoro17,kki1,kki2,kki3,MNRTT,baoxiang0,bertinoro57,bertinoro58,bertinoro58bis}; see also the recent survey \cite{ruz}. The main results and techniques are now also available in the monograph \cite{baoxiang}. Here we carry on this investigation. \par 
 To give a general setting for our results, let us denote by $T$ the linear operator providing the solution of a well-posed Cauchy problem for a partial differential equation, in suitable function spaces. Generically, we expect that $T$ is written in the form of pseudodifferential operator (PSDO), or Fourier integral operator (FIO). Both for theoretic and numeric purposes often it turns out necessary to decompose the initial datum with respect to a given frame and to reconstruct the solution by superposition, after studying the evolution of each wave packet. In this picture, the operator $T$  is then regarded as an infinite matrix, which for Gabor frames reads
\begin{equation}\label{intro3}
\langle T \pi(\mu)g,\pi(\lambda)g\rangle.
\end{equation}
The more this matrix is sparse, the more the above representation is tight.\par 
For example, let $T$ be a pseudodifferential operator with symbol $\sigma(x,\xi)$ in the class $S^0_{0,0}$, namely satisfying \begin{equation}\label{intro9}
|\partial^\alpha \sigma(z)|\leq C_{\alpha},\quad z=(x,\xi)\in\rdd.
\end{equation}
 Then, it was proved in \cite{charly06,GR,rochberg,tataru} that $T$ has a super-polynomial Gabor decay:
\begin{equation}\label{intro10}
|\langle T \pi(\mu)g,\pi(\lambda)g\rangle|\leq C_N (1+|\lambda-\mu|)^{-N}
\end{equation}
for every $N>0$, if the window $g$ is Schwartz.\par
Instead, when $T$ is a FIO with phase function of quadratic type, cf.\  \cite{AS78}, a similar result was shown in \cite{fio1,tataru}, with the difference $|\lambda-\mu|$ replaced by $|\lambda-\chi(\mu)|$, where $\chi$ is the corresponding canonical transformation. Applications were given there to Schr\"odinger equations with variable coefficients; see also \cite{CNG,fio5,fio3,QY,QY2010}. \par
For strictly hyperbolic equations with variable coefficients it is instead well-known \cite{cnr-flp,cnr-global} that the corresponding propagators, which are classical FIO with phase function homogeneous of degree 1 in the dual variables, cf.\  \cite[Vol.\ 4]{hormander}, do not have a sparse Gabor matrix; see also Remark \ref{osservazione-variabile} below. For these operators the ``correct'' wave packets are given by curvelet-type waves, tailored to a parabolic scaling, as shown in \cite{candes,B18,GT,guo-labate,Smith}. All these papers show a super-polynomial decay for the matrix of the propagator. 
\par
The present work is devoted to a systematic study in terms of Gabor frames of a general class of constant coefficient evolution operators. Besides the class of operators, the main novelty with respect to the existent literature is the investigation of super-exponential decay of the Gabor matrix (as opposite to super-polynomial).\par
To be more precise, we consider operators of the form 
\begin{equation}\label{op000}
P(\partial_t,D_x)=\partial_t^m +\sum_{k=1}^{m}a_k(D_x)\partial_t^{m-k},\quad t\in\R,\ x\in\R^d
\end{equation}
%Elena%%%%%%%%%%
where $a_{k}(\xi)$, $1\leq k\leq m$, are polynomials, whose degree may be arbitrary; as usual, $D_{x_j}=\frac{1}{2\pi i}\partial_{x_j}$, $j=1,\ldots,d$.
We suppose that the corresponding forward Cauchy problem is well-posed in a very mild sense, namely in $\cS(\rd)$. A necessary and sufficient condition for this is given by the forward Hadamard-Petrowsky condition, involving the complex roots $\tau\in\bC$ of the symbol $P(i\tau,\xi)$:
\begin{equation}\label{hp000}
(\tau,\xi)\in\mathbb{C}\times\rd,\quad P(i\tau,\xi)=0 \Longrightarrow {\rm Im}\, \tau\geq -C,
\end{equation}
for some constant $C>0$. When the polynomials $a_k(\xi)$ in \eqref{op000} have degree $\leq k$, this reduces to the definition of (weak) hyperbolicity, but we do not assume this here. There is a wide literature concerning constant coefficients operators, and the study of basic problems as hypoellipticity, fundamental solutions, etc., has reached the highest level of sophistication, with a combination of techniques from Algebraic Geometry and Mathematical Analysis \cite[Vol. II]{hormander}. 
Now, we will show that the Gabor matrix of the corresponding propagator enjoys a  {\it super-exponential}  decay, and we link the precise decay rate to the algebraic growth of the imaginary part of the roots $\tau$ of $P(i\tau,\zeta)$, as $\zeta\in\bC$, ${\rm Im}\, \zeta\to+\infty$. \par
 \par
 As a simple and typical model, let us consider the Cauchy problem for the wave equation:
 \begin{align}\label{intro4}
 &\partial^2_t u-\Delta_x u=0,\quad (t,x)\in\R\times\rd,\\ \nonumber
 & u(0,x)=u_0(x),\quad \partial_t u(0,x)=u_1(x).
 \end{align}
We may express the solution in the form
\begin{equation}\label{intro5}
u(t,x)=T_t u_1(x)+\partial_t T_t u_0(x),
\end{equation}
where $T_t$ is the Fourier multiplier
\begin{equation}
T_t f(x)=\int e^{2\pi i x\xi}\sigma_t(\xi) \widehat{f}(\xi)\, d\xi,
\end{equation}
with symbol
\begin{equation}\label{intro7}
\sigma_t(\xi)=\frac{\sin(2\pi|\xi|t)}{2\pi|\xi|},\quad \xi\in\R^d.
\end{equation}
Taking the Gaussian as window function $g$ one has in fact the Gaussian decay
 \begin{equation}\label{intro8}
 |\langle T_t \pi(\mu)g,\pi(\lambda)g\rangle|\leq C \exp\big(-\epsilon|\lambda-\mu|^2);
\end{equation}
(we address to the sequel of the paper for the dependence of $C$ and $\epsilon>0$ on $t$).
This is easily obtained from the explicit expression of the forward fundamental solution $E(t,x)=\Fur^{-1}_{\xi\to x} \sigma(t,\xi)$ (see Example \ref{esempio} below), but this approach does not extend to the more general equations above. Instead, we regard here the propagator $T_t$ as a PSDO where, with respect to the approach via Fourier integral operators, the $\xi$-dependent phase functions are absorbed into the symbol $\eqref{intro7}$. Because of its oscillations, $\sigma_t(\xi)$ then belongs the the non-standard H\"ormander's class $S^0_{0,0}$ defined by \eqref{intro9}, and therefore \eqref{intro10} holds for $T_t$.
Actually, seeking super-exponential decay, is is natural to introduce an analytic version of the class $S^0_{0,0}$, by assuming for $s\geq0$
\begin{equation}\label{intro11}
|\partial^\alpha \sigma(z)|\leq C^{|\alpha|+1}(\alpha!)^s,\quad z\in\rdd,
\end{equation}
for a constant $C>0$ depending on $\sigma$. We say that the symbol $\sigma$ is Gevrey when $s>1$, analytic when $s=1$, and ultra-analytic if $s<1$.  

%Let us observe that, when $s=1$, $\sigma$ extends to an analytic function in the strip $\{z+i\zeta\in\mathbb{C}^{2d}\}:\, |\zeta|<1/C\}$, whereas when $s<1$, $\sigma$ extents to an entire function $\sigma(z+i\zeta)$ satisfying $|\sigma(z+i\zeta)|\leq \tilde{C}\exp(\delta|\zeta|^{1/(1-s)})$, for some constants $\tilde{C},\delta>0$ (see e.g.\  \cite[Section 6.1]{NR}).\par
We will show, in fact, that a pseudodifferential operator $T$ with symbol satisfying \eqref{intro11} displays a matrix decay 
\begin{equation}\label{intro13}
 |\langle T \pi(\lambda)g,\pi(\mu)g\rangle|\leq C \exp\big(-\epsilon|\mu-\lambda|^r)
\end{equation}
for some $C,\epsilon>0$, with $r=\min\{2,1/s\}$, for suitable windows $g$ (see Theorem \ref{equivdiscr-cont} below). In the case of the Gevrey-analytic symbols, i.e.\ $s\geq1$ in \eqref{intro11} and $r\leq 1$ in \eqref{intro13}, the estimates follow from the results in \cite{GR}. Instead, in view of the applications, our attention will be mainly addressed to the ultra-analytic case, i.e.\ $0\leq s<1$ in \eqref{intro11}, $r>1$ in \eqref{intro13}. Incidentally we note that the class of ultra-analytic functions exhibits exotic behaviors (e.g.\  it is not closed by composition), which make it amazing in many respects; we will see an instance at once.  \par
In fact, it turns out that the symbol $\sigma_t(\xi)$ in \eqref{intro7} satisfies the estimate \eqref{intro11} with $s=0$. This is already a curious remark even in dimension $1$, because the function  $\sin (2\pi |\xi|t)$ satisfies those estimates, e.g.\  for $|\xi|\geq1$, but $|\xi|^{-1}$ does not, for any $s<1$. It is really surprising that a similar miraculous combination occurs for a wide class of equations of the form \eqref{op000}.\par To state the precise result, we refine the forward Hadamard-Petrowsky condition by requiring that the symbol $P(i\tau,\zeta)$ of the operator in \eqref{op000} satisfies 
\[
(\tau,\zeta)\in\mathbb{C}\times\mathbb{C}^d,\quad P(i\tau,\zeta)=0 \Longrightarrow {\rm Im}\, \tau\geq -C(1+|{\rm Im}\,\zeta|)^\nu,
\]
for some $C>0$, $\nu\geq 1$. Then we will prove that the propogator $Tu=E(t,\cdot)\ast u$, where $E(t,\cdot)$ is the forward fundamental solution, has symbol $\sigma(t,\cdot)=\widehat{E}(t,\cdot)$ satisfying \eqref{intro11} with $s=1-1/\nu$, and therefore \eqref{intro13} holds with $r=\min\{2, \nu/(\nu-1)\}$, for suitable windows (Theorem \ref{teo5.2}). We always have $r>1$, i.e.\ super-exponential decay.\par
As a special case, we get Gaussian decay ($r=2$) for {\it all} hyperbolic equations. As another example, we get $r=2k/(2k-1)$, $k\in\bN\setminus\{0\}$, for the generalized heat equation:
\begin{equation}\label{intro12}
\partial_t u+(-\Delta_x)^ku=0.
\end{equation}
Accordingly, every column or row of the Gabor matrix, rearranged in decreasing order, displays a similar decay, i.e.\ we obtain an {\it exponential-type sparsity}.\par As an easy byproduct, we have the continuity of the propagator on modulation spaces (\cite{grochenig,baoxiang} and Section 2.4 below), which for the wave and Klein-Gordon equations was already proved in \cite{bertinoro2,bertinoro3,bertinoro12,bertinoro58bis} by other methods.
Actually our result for hyperbolic equations is of particular interest when the operator is only {\it weakly} hyperbolic, where we are not aware of any almost-diagonalization result in the literature (even super-polynomial). In this respect, there is a very small intersection with the wide literature, mentioned above, of curvelet-type representations, which instead deals with variable coefficient {\it strictly} hyperbolic equations (and super-polynomial matrix decay). 

% \begin{itemize}
%\item we shall treat hyperbolic equations, not necessarily strictly hyperbolic, of any order and in an arbitrary number of space variables;
%\item our class of equation will include also parabolic equations;
%\item the decay of the Gabor matrix will be of super-exponential type.
 %\end{itemize}
\par
The plan of the paper is the following. In Section 2 we provide the necessary preliminaries, concerning Gelfand-Shilov spaces, time-frequency representations and modulation spaces. Section 3 is devoted to the study of the symbols in \eqref{intro11}. The corresponding PSDOs are considered in Section 4, where we prove \eqref{intro13} and give boundedness results in modulation and Gelfand-Shilov spaces. The applications to evolution equations are given in Section 5.

 \section{Preliminaries}
\subsection{Notations}
The Schwartz class is denoted by
$\sch(\Ren)$, the space of tempered
distributions by  $\sch'(\Ren)$.   We
use the brackets  $\la f,g\ra$ to
denote the extension to $\sch '
(\Ren)\times\sch (\Ren)$ of the inner
product $\la f,g\ra=\int f(t){\overline
{g(t)}}dt$ on $L^2(\Ren)$.%  , the involution
% $g^*$ is $g^*(t) = \overline{g(-t) }$
% and the inverse Fourier transform is
% ${\check  f}(\o)=\Fur^{-1}f (\o)={\hat
% {f}}(-\o)$.\par  We write  $dx\wedge d\xi=\sum_{j=1}^d dx_j\wedge
%   d\xi_j$
%  for the canonical symplectic
%  2-form.\par
%COMM: involution and inverse FT are not used.

Euclidean norm of $ x \in {\bR}^d $ is given by $ |x| = \left( x_1 ^2 + \dots +x_d
^2 \right) ^{1/2}, $ and $ \langle x \rangle = ( 1 + |x|^2 )^{1/2}.$ We write $xy=x\cdot y$ for  the scalar product on
$\Ren$, for $x,y \in\Ren$.\par
We adopt the usual multi-index notation and recall that if $\alpha\in\bN^d$ we have
\begin{equation}\label{2alfa}
\sum _{\beta \leq \alpha} {\alpha \choose \beta}=2^{|\a|},
\end{equation}
and
\begin{equation}\label{eqn:39}
|\a|!\leq d^{|\a|}\a!.
\end{equation}
 The Fourier
transform is normalized to be ${\hat
  {f}}(\o)=\Fur f(\o)=\int
f(t)e^{-2\pi i t\o}dt$.

Translation and modulation operators, $T$ and $M$ are defined by
$$
T_x f(\cdot) = f(\cdot - x) \;\;\; \mbox{ and } \;\;\;
 M_x f(\cdot) = e^{2\pi i x \cdot} f(\cdot), \;\;\; x \in {\bR}^d.
$$
For $z=(x,\xi)$ we shall also write
\[
\pi(z)f=M_{\xi} T_x f.
\]
The following relations hold
$$
M_y T_x  = e^{2\pi i x  y } T_x M_y, \;\;
 (T_x f)\hat{} = M_{-x} \hat f, \;\;
  (M_x f)\hat{} = T_{x} \hat f, \;\;\;
  x,y \in {\bR}^d,  f,g \in L^2 ({\bR}^d).
$$

For $0< p\leq\infty$ and a  weight $m$, the space $\ell^{p}_m
(\Lambda )$
 is  the
space of sequences $a=\{{a}_{\lambda}\}_{\lambda \in \Lambda }$
on a  lattice $\Lambda$, such that
$$\|a\|_{\ell^{p}_m}:=\left(\sum_{\lambda\in\Lambda}
|a_{\lambda}|^p m(\lambda)^p\right)^{1/p}<\infty
$$
(with obvious changes when $p=\infty$).

Throughout the paper, we
shall use the notation
$A\lesssim B$ to express the inequality
$A\leq c B$ for a suitable
constant $c>0$, and  $A
\asymp B$  for the equivalence  $c^{-1}B\leq
A\leq c B$.

\subsection{Gelfand-Shilov Spaces}

The Schwartz class $\cS(\rd)$ does not give quantitative information about how fast a function $f\in \cS(\rd)$ and its derivatives decay at infinity.
This is the main motivation to use subspaces of the Schwartz class, so-called Gelfand-Shilov type spaces, introduced in  \cite{GS}. Let us recall their definition and main properties, contained in \cite{GS,NR}.

\begin{definition} Let there be given $ s, r \geq0$ and $A,B>0$.
The Gelfand-Shilov type space $ S^{s,A} _{r,B} (\rd) $ is defined by
\begin{equation}\label{GFdef}
 S^{s,A} _{r,B} (\rd)=\{f\in\cS(\rd)\,:\,|x^\a\partial^\beta f(x)| \lesssim A^{|\a|}B^{|\beta|}(\a!)^r(\beta!)^s,\quad \a,\beta\in\bN^d\}.
\end{equation}
Their projective and inductive limits are denoted by
$$\Sigma^s_r={\rm proj}\lim_{A>0,B>0}S^{s,A} _{r,B};\quad S^s_r={\rm ind}\lim_{A>0,B>0}S^{s,A} _{r,B}.
$$
\end{definition}
The space $S^{s} _{r}(\rd) $ is nontrivial if and only if $ r + s > 1,$ or $ r + s = 1$ and $r, s > 0$.  So the smallest nontrivial space with $r=s$ is provided by $S^{1/2}_{1/2}(\rd)$. Every function of the type $P(x)e^{-a|x|^2}$, with $a>0$ and $P(x)$ polynomial on $\rd$, is in the class $S^{1/2}_{1/2}(\rd)$.  We observe the trivial inclusions $S^{s_1} _{r_1}(\rd)\subset S^{s_2} _{r_2}(\rd)$ for $s_1\leq s_2$ and $r_1\leq r_2$.
Moreover, if $f\in S^{s} _{r}(\rd)$, also $x^{\delta}\partial^\gamma f$ belongs to the same space for every fixed $\delta,\gamma$. \par The action of the Fourier transform on $S^{s} _{r}(\rd) $ interchanges the indices $s$ and $r$, as explained in the following theorem.

\begin{theorem} For $f\in \cS(\rd)$ we have $f\in S^{s} _{r}(\rd) $ if and only if $\hat{f}\in S^{r}_{s}(\rd).$
\end{theorem}

Therefore for $s=r$ the spaces $S^{s} _{s}(\rd)$ are invariant under the action of the Fourier transform.

%A characterization of the class $S^{s} _{r}(\rd)$ is detailed below.
%
\begin{theorem} \label{simetria} Assume $ s>0, r>0, s+r\geq1$. For $f\in  \cS(\rd)$, the following conditions are equivalent:
\begin{itemize}
\item[a)] $f \in  S^{s} _{r}(\rd)$ .
\item[b)]  There exist  constants $A, B>0,$ such that
$$ \| x^{\a}  f \|_{L^\infty} \lesssim A^{|\a|} (\a !) ^{r}  \quad\mbox{and}\quad  \| \o^{\b}  \hat{f} \|_{L^\infty} \lesssim  B^{|\b|} ( \b!) ^{s},\quad \a,\b\in \bN^d. $$
\item[c)]  There exist  constants $A, B>0,$ such that
$$ \| x^{\a}  f \|_{L^\infty} \lesssim A^{|\a|} (\a !) ^{r}  \quad\mbox{and}\quad  \| \partial^{\b}  f \|_{L^\infty} \lesssim  B^{|\b|} ( \b!) ^{s},\quad \a,\b\in \bN^d. $$

\item[d)] There exist  constants $h, k>0,$ such that
$$
 \|f  e^{h  |x|^{1/r}}\|_{L^\infty} < \infty \quad\mbox{and}\quad \| \hat f  e^{k |\o|^{1/s}}\|_{L^\infty} < \infty.$$
\end{itemize}
\end{theorem}

A suitable window class for weighted modulation spaces (see the subsequent Definition \ref{prva}) is the Gelfand-Shilov type space $\Sigma^1_1(\rd)$, consisting of  functions $f\in\cS(\rd)$ such that for {\it every}  constant $A>0$ and $B>0$
\begin{equation}\label{sigmadef}
|x^\a\partial^\beta f(x)| \lesssim A^{|\a|}B^{|\beta|}\a!\beta!,\quad \a,\beta\in\bN^d.
\end{equation}
We have $S^s_s(\rd)\subset\Sigma^1_1(\rd)\subset S^1_1(\rd)$ for every $s<1$.
Observe that the characterization of Theorem \ref{simetria} can be adapted to $\Sigma^1_1(\rd)$ by replacing the words ``there exist'' by ``for every'' and taking $r=s=1$.\par
Let us underline the following property, which exhibits  two equivalent ways of expressing the decay of a continuous function $f$ on $\rd$. This follows immediately from \cite[Proposition 6.1.5]{NR}, where the property was formulated for $f\in\cS(\rd)$. For the sake of clarity, we shall detail the proof showing the mutual dependence between the constants $\eps$ and $C$ below.
\begin{proposition}[\text{\cite[Proposition 6.1.5]{NR}}]\label{equi} Consider $r>0$ and let $h$ be a continuous function on $\rd$. Then the following conditions are equivalent:\\
(i) There exists a constant $\eps>0$ such that
\begin{equation}\label{eqn:A.7}
|h(x)|\lesssim \exp\big({-\eps |x|^{1/r}}\big),\quad x\in\rd,
\end{equation}
\noindent
(ii)  There exists a constant $C>0$ such that
\begin{equation}\label{powerdecay}
|x^\a h(x)|\lesssim C^{|\a|}(\a!)^{r},\quad x\in\rd, \,\a\in\bN^d.
\end{equation}
\end{proposition}

%\begin{proposition}\label{p:A.4}
%The following conditions are equivalent:
%\begin{enumerate}
%\item the condition (\ref{eqn:A.1}) holds, i.e.\ there exists a constant $\epsilon>0$ such that
%\begin{equation}\label{eqn:A.7}
%\abs{f(x)}\lesssim e^{-\epsilon\abs{x}^{\frac{1}{r}}},\quad   x\in \rd;
%\end{equation}
%\item there exists a constant $C>0$ such that
%\begin{equation}\label{eqn:A.8}
%\abs{x^\alpha f(x)}\lesssim C^{\abs{\alpha}}(\alpha!)^r,\quad   x\in \rd,\ \alpha\in \mathbb{N}^d.
%\end{equation}
%\end{enumerate}
%\end{proposition}
\begin{proof}
We  re-write (\ref{eqn:A.7}) in the form
\begin{equation}\label{eqn:A.9}
\abs{h(x)}^{\frac{1}{r}} \lesssim \exp\big({-\frac{\epsilon}{r}\abs{x}^{\frac{1}{r}}}\big),\quad    x\in \rd.
\end{equation}
In turn, (\ref{eqn:A.9}) can be re-written as
\begin{equation}\label{eqn:A.10}
\sup_{x\in\rd}\sum_{n=0}^{\infty}\left(\frac{\epsilon}{r}\right)^n(n!)^{-1} \abs{x}^{\frac{n}{r}}\abs{h(x)}^{\frac{1}{r}}<\infty.
\end{equation}
Hence the
sequence of the terms of the
series is uniformly bounded,
as well as the sequence of
the $r$-th powers:
\[
\frac{\epsilon^{r n}}{r^{r n}}(n!)^{-r}\abs{x}^n\abs{h(x)},\quad n\in\bN,
\]
and we obtain
\[
\abs{x}^n\abs{h(x)}\lesssim \frac{r^{r n}}{\epsilon^{r n}}(n!)^r,\quad    x\in \rd,\ n=0,1,\dots.
\]
Hence, writing $\abs{\alpha}=n$ and applying (\ref{eqn:39}):
\[
\abs{x^\alpha h(x)}\lesssim \left({r d/\epsilon}\right)^{r|\a|}(\a!)^r = C^{\abs{\alpha}}(\alpha!)^r, \quad   x\in \rd,\ \alpha\in \mathbb{N}^d,
\]
where $C=\displaystyle{\left( {r d/\epsilon}\right)^r}$. Therefore (\ref{powerdecay}) is proved.\par
In the opposite direction, let (\ref{powerdecay}) be satisfied. Using the following inequalities
$$\abs{x}^n\leq \sum_{|\a|=n}\frac{n!}{\a!}|x^\a|,\quad \sum_{|\a|=n}\frac{n!}{\a!}=d^n,\quad\mbox{and}\quad \a!\leq |\alpha|!,$$ and the assumption \eqref{powerdecay} we obtain
\[
\abs{x}^n\abs{h(x)}\leq \sum_{\abs{\alpha}=n}\frac{n!}{\a!}\abs{x^\alpha h(x)}\lesssim\sum_{\abs{\alpha}=n}\frac{n!}{\a!}C^{|\a|}(\a!)^{r}\leq
C^n(n!)^r  d^n= (d C)^n (n!)^r,
\]
for every $x\in\rd$, $n\in\bN$.  Therefore the sequence
\[
(d C)^{-n}(n!)^{-r}\abs{x}^n\abs{h(x)},   \quad n\in\bN
\]
is uniformly bounded for $x\in \rd$, as well as the sequence
\[
(d C)^{-\frac{n}{r}}(n!)^{-1}\abs{x}^{\frac{n}{r}}\abs{h(x)}^{\frac{1}{r}}, \quad n\in\bN.
\]
If we choose $\epsilon= q (d C)^{-\frac{1}{r}}$, for a fixed $q\in (0,1)$, we conclude
\[
e^{\epsilon\abs{x}^{\frac{1}{r}}}\abs{h(x)}^{\frac{1}{r}}=\sum_{n=0}^{\infty} q^n(d C)^{-\frac{n}{r}}(n!)^{-1}\abs{x}^{\frac{n}{r}}\abs{h(x)}^{\frac{1}{r}}\lesssim \sum_{n=0}^{\infty} q^n.
\]
This is  (\ref{eqn:A.9}); hence we get (\ref{eqn:A.7}) and the proof is complete.
\end{proof}

It follows from this proof the  precise relation  between the constants  $\eps$ and $C$. Indeed,
assuming \eqref{eqn:A.7}, then \eqref{powerdecay} is satisfied with $C=\displaystyle{\left(r d/\eps\right)^r}$. Viceversa, \eqref{powerdecay}
implies \eqref{eqn:A.7} for any $\epsilon< r(d C)^{-{1}/{r}}$. The bound is sharp for $d=1$. Also, it follows from the proof that the constant implicit in the notation $\lesssim$ in \eqref{eqn:A.7} depends only on the corresponding one in \eqref{powerdecay} and viceversa.  \par

The strong dual spaces of $S^s_r(\rd)$ and  $\Sigma^1_1(\rd)$  are called spaces of tempered ultra-distributions and denoted by $(S^s_r)'(\rd)$ and $(\Sigma^1_1)'(\rd)$, respectively. Notice that they contain the space of tempered distribution $\cS'(\rd)$.\par
%RIPETIZIONE CON FOGLIO MANOSCRITTO LUIGI:
The spaces $S^s_r(\rd)$ are nuclear spaces \cite{mitjagin}, and this property yields a corresponding kernel theorem; cf.\ \cite{treves}. \begin{theorem}\label{kernelT} There exists an isomorphism between the space of linear continuous maps $T$ from $S^s_r(\rd)$ to $(S^s_r)'(\rd)$ and $(S^s_r)'(\rdd)$, which associates to every $T$ a kernel $K_T\in (S^s_r)'(\rdd)$ such that $$\la Tu,v\ra=\la K_T, v\otimes \bar{u}\ra,\quad \forall u,v \in S^s_r(\rd).$$ $K_T$ is called the kernel of $T$. \end{theorem}

\subsection{Time-frequency representations.}  We recall the basic
concepts  of \tfa\ and  refer the  reader to \cite{grochenig} for the full
details. %  a complete
% introduction of time-frequency analysis.
Consider a distribution $f\in\cS '(\rd)$
and a Schwartz function $g\in\cS(\rd)\setminus\{0\}$ (the so-called
{\it window}).
The short-time Fourier transform (STFT) of $f$ with respect to $g$ is $V_gf (z) = \langle f, \pi (z)g\rangle
$, $z=(x,\xi)\in\rd\times\rd$. The  \stft\ is well-defined whenever  the bracket $\langle \cdot , \cdot \rangle$ makes sense for
dual pairs of function or (ultra-)distribution spaces, in particular for $f\in
\cS ' (\rd )$ and $g\in \cS (\rd )$, $f,g\in\lrd$, $f\in
(\Sigma_1^1) ' (\rd )$ and $g\in \Sigma_1^1 (\rd )$ or $f\in
(S^s_r) ' (\rd )$ and $g\in S^s_r (\rd )$.
Let us recall the covariance formula for the \stft\  that will be used in
the sequel
  \begin{equation}
    \label{eql2}
   V_{ {g}}(M_\eta T_y{f})(x,\xi)  = e^{-2\pi
     i(\xi-\eta)y}(V_gf)(x-y,\xi-\eta), \qquad x, y,\o , \eta \in \Ren.
  \end{equation}

Another time-frequency representation we shall use is the (cross-)Wigner distribution of $f,g\in \lrd$, defined as
\begin{equation}\label{cross-wigner}
W(f,g)\phas=\intrd f\Big(x+\frac t2\Big){\overline{g\Big(x-\frac t2\Big)}} e^{-2\pi i t\o}\,dt.
\end{equation}
If we set $\breve{g}(t)=g(-t)$, then the relation between cross-Wigner distribution and short-time Fourier transform is provided by
\begin{equation}\label{wignerSTFT}
W(f,g)\phas=2^d e^{4\pi i x\o}V_{\breve{g}}f(2x,2\o).
\end{equation}

\par

For the discrete description of function spaces and operators we use
Gabor frames. Let $\Lambda=A\zdd$  with $A\in GL(2d,\R)$ be a lattice
of the time-frequency plane.
 The set  of
time-frequency shifts $\G(g,\Lambda)=\{\pi(\lambda)g:\
\lambda\in\Lambda\}$ for a  non-zero $g\in L^2(\rd)$ is called a
Gabor system. The set $\G(g,\Lambda)$   is
a Gabor frame, if there exist
constants $A,B>0$ such that
\begin{equation}\label{gaborframe}
A\|f\|_2^2\leq\sum_{\lambda\in\Lambda}|\langle f,\pi(\lambda)g\rangle|^2\leq B\|f\|^2_2,\qquad \forall f\in L^2(\rd).
\end{equation}
 If \eqref{gaborframe} is satisfied, then there exists a dual window $\gamma\in L^2(\rd)$, such that $\cG(\gamma,\Lambda)$ is a frame, and every $f\in L^2(\rd)$ possesses the frame expansions
 \[
 f=\sum_{\lambda\in\Lambda}\langle f,\pi(\lambda)g\rangle\pi(\lambda)\gamma=\sum_{\lambda\in\Lambda}\langle f,\pi(\lambda)\gamma\rangle \pi(\lambda)g
 \]
 with unconditional convergence in $L^2(\rd)$. In conclusion, we list some results about time-frequency analysis of Gelfand-Shilov functions, cf.\  \cite{elena07, medit,GZ,T2}:
 \begin{align}\label{zimmermann1}
 f,g\in S^s_s(\rd),\ s\geq1/2 &\Rightarrow V_g f\in S^s_s(\rdd),\\
 \label{wigner-gelfand}
 f,g\in S^s_s(\rd),\ s\geq1/2&\Rightarrow W(f,g)\in S^s_s(\rdd),\\
\label{wigner-gelfand2}
 f,g\in \Sigma_1^1(\rd) &\Rightarrow W(f,g)\in \Sigma^1_1(\rdd).
 \end{align}
 Finally, if $g\in S^s_s(\rd$), $s\geq1/2$, then
 \begin{equation}\label{zimmermann2}
   f\in S^s_s(\rd)\Longleftrightarrow |V_g(f)(z)|\lesssim \exp\big({-\epsilon |z|^{1/s}}\big)\ \mbox{for some} \,\,\epsilon>0.
 \end{equation}

\subsection{Modulation Spaces}
Weighted modulation spaces measure the decay of the STFT on the time-frequency (phase space) plane and were introduced by Feichtinger in the 80's \cite{F1} for weight of sub-exponential growth at infinity. The study of weights of exponential growth was developed in \cite{medit,T2}.

\emph{Weight Functions.}  In the sequel $v$ will always be a
continuous, positive,  even, submultiplicative   function
(submultiplicative weight), i.e., $v(0)=1$, $v(z) =
v(-z)$, and $ v(z_1+z_2)\leq v(z_1)v(z_2)$, for all $z,
z_1,z_2\in\rd.$  Submultiplicativity implies that $v(z)$ is \emph{dominated} by an exponential function, i.e.
\begin{equation} \label{weight}
 \exists\, C, k>0 \quad \mbox{such\, that}\quad  1\leq v(z) \leq C e^{k |z|},\quad z\in \rd.
\end{equation}

For example, every weight of the form
\begin{equation} \label{BDweight}
v(z) =   e^{s|z|^b} (1+|z|)^a \log ^r(e+|z|)
\end{equation}
 for parameters $a,r,s\geq 0$, $0\leq b \leq 1$ satisfies the
above conditions.\par

We denote by $\mathcal{M}_v(\rd)$ the space of $v$-moderate weights on $\rd$;
these  are measurable positive functions $m$ satisfying $m(z+\zeta)\leq C
v(z)m(\zeta)$ for every $z,\zeta\in\rd$.

\begin{definition}  \label{prva}
Given  $g\in\Sigma^1_1(\rd)$, a  weight
function $m\in\cM _v(\rdd)$, and $1\leq p,q\leq
\infty$, the {\it
  modulation space} $M^{p,q}_m(\Ren)$ consists of all tempered
ultra-distributions $f\in(\Sigma^1_1)' (\rd) $ such that $V_gf\in L^{p,q}_m(\Renn )$
(weighted mixed-norm spaces). The norm on $M^{p,q}_m(\rd)$ is
\begin{equation}\label{defmod}
\|f\|_{M^{p,q}_m}=\|V_gf\|_{L^{p,q}_m}=\left(\int_{\Ren}
  \left(\int_{\Ren}|V_gf(x,\o)|^pm(x,\o)^p\,
    dx\right)^{q/p}d\o\right)^{1/q}  \,
\end{equation}
(obvious changes if $p=\infty$ or $q=\infty$).
\end{definition}
We observe that for $f,g\in \Sigma^1_1(\rd)$ the above integral is convergent and thus $\Sigma^1_1(\rd)\subset M^{p,q}_m(\rd) $, $1\leq p,q \leq\infty$, cf.\ \ \cite{medit}, with dense inclusion when $p,q<\infty$, cf.\ \ \cite{elena07}. When $p=q$, we simply write $M^{p}_m(\rd)$ instead of $M^{p,p}_m(\rd)$. The spaces $M^{p,q}_m(\rd)$ are Banach spaces and every nonzero $g\in M^{1}_v(\rd)$ yields an equivalent norm in \eqref{defmod} and so $M^{p,q}_m(\Ren)$ is independent on the choice of $g\in  M^{1}_v(\rd)$.
\par We also recall the inversion formula for
the STFT (see  \cite[Proposition 11.3.2]{grochenig} and \cite[Proposition 2.6]{medit} for exponential weights): assume $g\in M^{1}_v(\rd)\setminus\{0\}$,
 $f\in M^{p,q}_m(\rd)$, then
\begin{equation}\label{invformula}
f=\frac1{\|g\|_2^2}\int_{\R^{2d}} V_g f(x,\o)M_\o T_x g\, dx\,d\o,
\end{equation}
and the  equality holds in $M^{p,q}_m(\rd)$ (observe that $M^{1}_v(\rd)\subset M^2(\rd)=L^2(\rd)$, so $\|g\|_2<\infty$).
\par
\section{Gevrey-analytic and ultra-analytic symbols}\label{tfc} In this section we characterize the smoothness and the growth of a function $f$ on $\rd$ in terms of the decay of its STFT $V_g f$, for a suitable window $g$. In the proofs we shall detail the relations among the constants $C_f,C_g,C_{f,g}$ and $\epsilon$ which come into play, since it could be useful for numerical purposes.
\begin{theorem}\label{teo1}
Consider $s>0$, $m\in \mathcal{M}_v(\rd)$, $g\in M^1_{v\otimes 1}(\rd)\setminus\{0\}$ such that for some $C_g>0$,
\begin{equation}\label{finestra}
\|\partial^\a g\|_{L^1_v(\rd)}\lesssim C_g^{|\a|}(\a!)^s,\quad\a\in\bN^d.
\end{equation}
For $f\in\cC^\infty(\rd)$ the following conditions are equivalent:\\
(i) There exists a constant $C_f>0$ such that
\begin{equation}\label{smoothf}
|\partial^\a f(x)|\lesssim m(x)C_f^{|\a|}(\a!)^s,\quad x\in\rd,\,\a\in\bN^d.
\end{equation}
\noindent
(ii)  There exists a constant $C_{f,g}>0$ such that
\begin{equation}\label{STFTf}
|\o^\a V_gf\phas|\lesssim m(x)C_{f,g}^{|\a|}(\a!)^s,\quad \phas\in\rdd,\,\a\in\bN^d.
\end{equation}
(iii)  There exists a constant $\eps>0$ such that
\begin{equation}\label{STFTeps}
|V_gf\phas|\lesssim m(x)\exp\big({-\eps|\o|^{1/s}}\big),\quad \phas\in\rdd,\,\a\in\bN^d.
\end{equation}
\end{theorem}
If the equivalent conditions \eqref{smoothf}, \eqref{STFTf}, \eqref{STFTeps} are satisfied, we say that $f$ is a Gevrey symbol when $s>1$, an analytic symbol if $s=1$ and an ultra-analytic symbol when $s<1$.
\begin{proof}
$(i)\Rightarrow (ii)$. We can use Leibniz' formula and write
$$\displaystyle
\partial^{\alpha} (f T_xg) (t) = \sum _{\beta \leq \alpha} {\alpha \choose \beta}
\partial^{\alpha - \beta} f (t) T_x(\partial^{\beta} g) (t),\quad t\in\rd,\,\a\in\bN^d.
$$
Let us estimate the $L^1(\rd)$-norm of the products $\partial^{\alpha - \beta} f  T_x(\partial^{\beta} g)$. Using the positivity and $v$-moderateness of the weight $m$, H\"older inequality and the assumptions \eqref{smoothf} and \eqref{finestra},
\begin{align*}
\|\partial^{\alpha - \beta} f T_x(\partial^{\beta} g)\|_{L^1}&\leq \|\partial^{\alpha - \beta} f\|_{L^\infty_{1/m}}\|m T_x(\partial^{\beta} g)\|_{L^1}\\
&\lesssim C_f^{|\a-\b|}((\a-\b)!)^s \|v(\cdot-x)\partial^\b g(\cdot-x)\|_{L^1} m(x)\\
&\lesssim C_f^{|\a-\b|}((\a-\b)!)^s C_g^{|\b|}(\b!)^s m(x)\\
&\leq C^{|\a|} (\a!)^s m(x)
\end{align*}
where $C:=\max\{ C_f, C_g\}$ and we used $(\a-\b)! \b!\leq \a!$. These estimates tell us, in particular, that $V_g f\phas=\cF( f T_x g)(\o)$ is well-defined. We can exchange partial derivatives and Fourier transform as follows
\begin{equation*}
\o^\a V_g f\phas=\frac1{(2\pi i)^{|\a|}}\cF(\partial^\a (f T_x g))(\o)= \frac1{(2\pi i)^{|\a|}}\sum_{\b\leq\a} {\alpha \choose \beta} \cF(\partial^{\alpha - \beta} f T_x(\partial^{\beta} g))(\o).
\end{equation*}
Using $$|\cF(\partial^{\alpha - \beta} f T_x(\partial^{\beta} g))(\o)|\leq \|\cF(\partial^{\alpha - \beta} f T_x(\partial^{\beta} g))\|_{L^\infty}\leq \|\partial^{\alpha - \beta} f T_x(\partial^{\beta} g)\|_{L^1},$$
the majorizations above and \eqref{2alfa},
\begin{equation*}
|\o^\a V_g f\phas|\lesssim\frac{m(x)}{(2\pi)^{|\a|}}\sum _{\beta \leq \alpha} {\alpha \choose \beta}C^{|\a|}(\a!)^{s}
= \frac{m(x)}{\pi^{|\a|}}C^{|\a|}(\a!)^{s}.
\end{equation*}
This proves \eqref{STFTf}, with constant $\displaystyle C_{f,g}:=\frac{\max\{ C_f, C_g\}}{\pi}$.\par
$(ii)\Rightarrow (i)$. We use the inversion formula \eqref{invformula}, observing  that \eqref{STFTf} implies $f\in M^{\infty,1}_{m\otimes1}(\rd)$, so that the equality in \eqref{invformula} holds a.e. and we can assume that it holds everywhere since $f$ is smooth.  Let us  consider the partial derivatives of $f$ and exchange them with the integrals, the estimates below will provide the justification of this operation.
So, formally, we can write
\begin{equation*}
\partial^\a f(t)=\frac1{\|g\|_2^2}\int_{\R^{2d}} V_g f(x,\o) \partial^\a (M_\o T_x g)(t)\, dx\,d\o,\quad t\in\rd.
\end{equation*}
Using Leibniz' formula $ \partial^\a (M_\o T_x g)=\sum _{\beta \leq \alpha} {\alpha \choose \beta}(2\pi i \o)^{\b} M_{\o} T_x(\partial^{\a-\b}g)$
we estimate
\begin{equation*}
|\partial^\a f(t)|\leq \frac1{\|g\|_2^2}\sum _{\beta \leq \alpha} {\alpha \choose \beta}(2\pi)^{|\b|}   \int_{\R^{2d}}| V_g f(x,\o)\o^\b|\,\cdot |T_x\partial^{\a-\b} g(t)|\, dx\,d\o,\quad t\in\rd.
\end{equation*}
We set
$$I_{\a,\b}(t):= \int_{\R^{2d}}| V_g f(x,\o)\o^\b|\,\cdot |T_x\partial^{\a-\b} g(t)|\, dx\,d\o
$$
and prove that, for every fixed $t\in\rd$,   $I_{\a,\b}(t)$ are absolutely convergent integrals. Using $1+|\o|^{d+1}\leq c_d\sum_{|\gamma|\leq d+1}|\o^\gamma|$, where $c_d$ depends only on the dimension $d$, and  assumption \eqref{STFTf},
\begin{align*}
I_{\a,\b}(t) &= \int_{\R^{2d}}| V_g f(x,\o)\o^\b|\,\frac{1+|\o|^{d+1}}{1+|\o|^{d+1}} |T_x\partial^{\a-\b} g(t)|\, dx\,d\o\\
&\lesssim\int_{\R^{2d}} \sum_{|\gamma|\leq d+1} | V_g f(x,\o)\o^{\b+\gamma}|\,\frac{1}{1+|\o|^{d+1}} |T_x\partial^{\a-\b} g(t)|\, dx\,d\o\\
&\lesssim \sum_{|\gamma|\leq d+1} C_{f,g}^{|\b+\gamma|}((\b+\gamma)!)^s \left(\int_{\R^{d}} \frac{1}{1+|\o|^{d+1}}\,d\o \right) \left(\int_{\R^{d}} m(x)\, |T_x\partial^{\a-\b} g(t)|\, dx\right).
\end{align*}
Now, $(\b+\gamma)!\leq 2^{|\b|+|\gamma|}(\b!)(\gamma!)\leq 2^{d+1}(d+1)! 2^{|\b|}\b!$ and $C_{f,g}^{|\b+\gamma|}=C_{f,g}^{d+1}C_{f,g}^{|\b|}$ so, forgetting about the constants depending only on $d$, $f$ and $g$,
\begin{align*}
I_{\a,\b}(t) &\lesssim  C_{f,g}^{|\b|}2^{s\b}(\b!)^s \left(\int_{\R^{d}} m(x)\, |T_x\partial^{\a-\b} g(t)|\, dx\right)\\
&\lesssim  C_{f,g}^{|\b|}2^{s\b}(\b!)^s m(t) \left(\int_{\R^{d}} v(x-t) |\partial^{\a-\b} g(t-x)|\, dx\right)\\
&= C_{f,g}^{|\b|}2^{s\b}(\b!)^s m(t) \|\partial^{\a-\b} g\|_{L^1_v}.
\end{align*}
The estimate above and the assumption \eqref{finestra} on $g$ allow the following majorization:
\begin{align*}
|\partial^\a f(t)|&\lesssim \frac{m(t)}{\|g\|_2^2}\sum _{\beta \leq \alpha} {\alpha \choose \beta}(2\pi)^{|\b|} (2^{s}C_{f,g})^{|\b|}(\b!)^s  C_g^{|\a-\b|}((\a-\b)!)^s\\
&\leq \frac{m(t)}{\|g\|_2^2}(\a!)^s \max\{C_{f,g},C_g\}^{|\a|} 2^{s|\a|}(2\pi)^{|\a|} 2^{|\a|}\\
&\lesssim m(t) C_f^\a (\a!)^s,
\end{align*}
where $C_f:=2^{s+2}\pi\max\{C_{f,g},C_g\}$, and we used $(\b)!(\a-\b)!\leq\a!$ and \eqref{2alfa}.\par
\noindent
$(ii)\Leftrightarrow (iii)$. The equivalence is an immediate consequence of Proposition \ref{equi} and the subsequent remarks on the constants, where, for every fixed $x\in\rd$,  we choose $h(\o):=V_g f\phas/m(x)$ and $r=s$.
\end{proof}

A natural question is whether we may find window functions satisfying \eqref{finestra}. To this end, we recall the following characterization of Gelfand-Shilov spaces.
\begin{proposition}\label{pro3.2} Let $g\in\cS(\rd)$. We have $g\in S^s_r(\rd)$, with $s,r>0$, $r+s\geq1$, if and only if there exist constants $A>0$, $\epsilon>0$ such that
$$|\partial^\a g(x)|\lesssim A^{|\a|} (\a!)^s \exp\big({-\eps|x|^{1/r}}\big),\quad x\in\rd,\,\,\a\in\bN^d.
$$
We have $g\in \Sigma^1_1(\rd)$ if and only if, for every $A>0$, $\epsilon>0$,
$$|\partial^\a g(x)|\lesssim A^{|\a|} \a! \exp\big({-\eps|x|}\big),\quad \, x\in\rd,\,\,\a\in\bN^d.
$$
\end{proposition}
\begin{proof} The first part of the statement is in \cite[Proposition 6.1.7]{NR}. For the second part, the assumption $g\in\Sigma^{1}_1(\rd)$ means that for every $A>0, B>0$,
$$|x^\b\partial^\a g(x)|\lesssim A^{|\a|}B^{|\b|} \a! \b!,\quad x\in\rd,\,\,\a,\,\b\in\bN^d.
$$
Therefore the function  $h=\displaystyle{\frac{\partial^\a g}{A^{|\a|}\a!}}$ satisfies \eqref{powerdecay} in Proposition \ref{equi} for every $C>0$, and the estimate \eqref{eqn:A.7} is then satisfied for every $\eps>0$ (see the observations after the proof of Proposition \ref{equi} for the uniformity of the constants). This gives the claim.
\end{proof}

Hence every $g\in S^s_r(\rd)$ with $s>0$, $0<r<1$, $s+r\geq 1$, satisfies \eqref{finestra} for every submultiplicative weight $v$ (see \eqref{weight}). The same holds true if $g\in\Sigma_1^1(\rd)$ and $s\geq1$.
\section{Almost Diagonalization for Pseudodifferential operators}\label{section4}
In this section we extend some results of \cite{GR} to the case of pseudodifferential operators having ultra-analytic symbols. Our result covers
also the Gevrey-analytic case which was already discussed in \cite{GR}.\par
The Weyl form $\sigma^w$ of a  pseudodifferential operator (the so-called Weyl operator or Weyl transform) with symbol $\sigma\phas$ on $\rdd$ is formally defined by
\begin{equation}\label{Weyl}
\sigma^w f (x)=\intrd \sigma\left(\frac{x+y}2,\xi\right)e^{2\pi i (x-y)\xi}f(y)\,dy d\xi.
\end{equation}
Using the Wigner distribution \eqref{cross-wigner}, we can recast the definition as follows
\begin{equation}\label{Weyl2}
\langle \sigma^w f,g\rangle=\langle \sigma, W(f,g)\rangle.
\end{equation}
Since for $f,g\in\cS(\rd)$ we have $W(f,g)\in \cS(\rdd)$, \eqref{Weyl2} defines for $\sigma\in \cS'(\rdd)$ a continuous map $\sigma^w:\cS(\rd)\to\cS'(\rd)$. The Schwartz kernel $K$ of $\sigma^w$ is given by
\begin{equation}\label{nucleo1}
K(x,y)=\Fur^{-1}_{\xi\to x-y} \sigma\Big(\frac{x+y}{2},\xi\Big)\in\cS'(\rdd).
\end{equation}
On the other hand, in view of the kernel theorem in $\cS-\cS'$, every linear continuous map from $\cS(\rd)$ to $\cS'(\rdd)$ can be represented by means of a kernel $K\in \cS'(\rdd)$, hence as in \eqref{Weyl}, \eqref{Weyl2} with symbol
\begin{equation}\label{nucleo2}
\sigma(x,\xi)=\Fur_{y\to\xi} K\Big(x+\frac{y}{2},x-\frac{y}{2} \Big).
\end{equation}
Same arguments are valid in Gelfand-Shilov classes with $s=r\geq1/2$ (see \cite{mitjagin,treves}). Namely, considering $\sigma\in (S^s_s)'(\rdd)$, $f,g\in S^s_s(\rd)$ in \eqref{Weyl}, \eqref{Weyl2}, we have $\sigma^w: S^s_s(\rd)\to (S^s_s)'(\rd)$ continously, in view of \eqref{wigner-gelfand}. The kernel $K$ of $\sigma^w$, given by \eqref{nucleo1}, belongs to $(S^s_s)'(\rdd)$. In the opposite direction, in view of Theorem \ref{kernelT}, every linear continuous map from $S^s_s(\rd)$ to $(S^s_s)'(\rd)$ can be represented in the form \eqref{Weyl}, \eqref{Weyl2}, with $\sigma\in (S^s_s)'(\rdd)$ given by \eqref{nucleo2}. The same holds for the couple of spaces $\Sigma^1_1,(\Sigma^1_1)'$.\par
The crucial relation between the action of the Weyl operator $\sigma^w$ on time-frequency shifts and the short-time Fourier transform of its symbol, contained in \cite[Lemma 3.1]{charly06} can now be extended to Gelfand-Shilov spaces  and their dual spaces as follows.
\begin{lemma}\label{lemma41} Consider $s\geq1/2$, $g\in S^s_s(\rd)$, $\Phi=W(g,g)$. Then, for $\sigma\in (S^s_s)'(\rdd)$,
\begin{equation}\label{311}
|\la\sigma^w \pi(z)g,\pi(w) g\ra|=\left|V_\Phi \sigma\left(\frac{z+w}2,j(w-z)\right)\right|=|V_\Phi\sigma(u,v)|
\end{equation}
and
\begin{equation}\label{312}
|V_\Phi \sigma(u,v)|=\left|\la\sigma^w \pi\left(u-\frac12 j^{-1}(v)\right)g,\pi\left(u+\frac12 j^{-1}(v)\right) g\ra\right|,
\end{equation}
where $j(z_1,z_2)=(z_2,-z_1)$.
Moreover, the same results hold true if we replace the space $S^s_s(\rd)$ with the space $\Sigma^1_1(\rd)$.
\end{lemma}
\begin{proof}
Since $\Phi=W(g,g)\in S^s_s(\rdd)$ for $g\in S^s_s(\rd)$ the duality $\la \sigma,\pi(u,v)\Phi\ra_{(S^s_s)'\times S^s_s}$ is well-defined so that the short-time Fourier transform $V_\Phi \sigma(u,v)$ makes sense. The same pattern applies to the case $g\in\Sigma_1^1(\rd)$. The rest of the proof is analogous to \cite[Lemma 3.1]{charly06}.
\end{proof}

We now exhibit a  characterization of the operator $\sigma^w$ in terms of its continuous Gabor matrix. The symbol $\sigma$ belongs to the classes defined in Section 3, the dimension being now $2d$.
\begin{theorem}\label{CGelfandPseudo}  Let $s\geq1/2$, and $m\in\mathcal{M}_v(\rdd)$. If $1/2\leq s<1$ consider a window function $g\in S^s_s(\rd)$; otherwise, if $s\geq1$, assume either $g\in \Sigma^1_1(\rd)$, or $g\in S^s_s(\rd)$ and the following growth condition on the weight $v$:
\[
v(z)\lesssim \exp\big({\epsilon|z|^{1/s}}\big),\quad z\in \rdd,
\]
for every $\epsilon>0$. Then the following properties are equivalent for $\sigma\in\cC^\infty(\rdd)$:
\par {
(i)} The symbol $\sigma$ satisfies
\begin{equation}\label{simbsmooth} |\partial^\a \sigma(z)|\lesssim m(z) C^{|\a|}(\a!)^{s}, \quad \forall\, z\in\rdd,\,\forall \a\in\bN^{2d}.\end{equation}
{(ii)} There exists $\eps>0$ such that
\begin{equation}\label{unobis2s} |\langle \sigma^w \pi(z)
g,\pi(w)g\rangle|\lesssim m\left(\frac{w+z}2\right)\exp\big({-\eps|w-z|^{1/s}}\big),\qquad \forall\,
z,w\in\rdd.
\end{equation}
\end{theorem}
\begin{proof}
$(i)\Rightarrow (ii)$. Proposition \ref{pro3.2} applied to the window $\Phi=W(g,g)$ in Lemma \ref{lemma41}, which lives in the space $S^s_s(\rdd)$ since $g\in S^s_s(\rd)$ by \eqref{wigner-gelfand}, and the assumptions on $v$ (recall also \eqref{weight}) imply that $\Phi$ satisfies the assumptions of Theorem \ref{teo1}. Hence, using the equivalence \eqref{smoothf} $\Leftrightarrow$ \eqref{STFTeps}, the  assumption \eqref{simbsmooth} is equivalent to the following decay estimate of the corresponding short-time Fourier transform
$$|V_\Phi \sigma(u,v)|\lesssim m(u) e^{-\eps |v|^{\frac1s}}, \quad u,v\in\rdd,
$$
for a suitable $\eps>0$, which combined with   \eqref{311} yields
$$|V_\Phi \sigma\big(\frac{z+w}2,j(w-z)\big)|\lesssim m\left(\frac{w+z}2\right) e^{-\eps |j(w-z)|^{\frac1s}}= m\left(\frac{w+z}2\right)e^{-\eps |w-z|^{\frac1s}}
$$
that is $(ii)$.\par
$(ii)\Rightarrow (i)$. Relation \eqref{312} and the decay assumption \eqref{unobis2s} give
\begin{align*}
|V_\Phi \sigma(u,v)|&=\left|\la\sigma^w \pi\left(u-\frac12 j^{-1}(v)\right)g,\pi\left(u+\frac12 j^{-1}(v)\right)g\ra\right|\\
&\lesssim m(u) e^{-\eps | j^{-1}(v)|^{\frac1s}}=m(u)e^{-\eps | v|^{\frac1s}}
\end{align*}
and using the equivalence  \eqref{smoothf} $\Leftrightarrow$ \eqref{STFTeps} we obtain the claim.
\end{proof}
\par
Of course, from \eqref{unobis2s} we deduce the discrete Gabor matrix decay in \eqref{unobis2discr} below. The viceversa requires more effort.
 Indeed, we appeal to a recent result obtained by   Gr{\"o}chenig and  Lyubarskii in \cite{GL09}.  They find sufficient conditions on the lattice $\Lambda=A \bZ^2$, $A\in GL(2,\R)$,  such that $g=\sum_{k=0}^n c_k H_k$, with $H_k$ Hermite function, forms
a  so-called Gabor (super)frame $\G(g,\Lambda)$. Besides they prove the existence of dual windows $\gamma$ that belong to the space $S^{1/2} _{1/2} (\R)$ (cf.\  \cite[Lemma 4.4]{GL09}).
This theory transfers to the $d$-dimensional case  by taking a tensor product $g=g_1\otimes\cdots\otimes g_d\in S^{1/2} _{1/2} (\rd)$ of windows as above, which defines a Gabor frame on the lattice $\Lambda_1\times\cdots\times\Lambda_d$ and   possesses a dual window $\gamma=\gamma_1\otimes\cdots\otimes \gamma_d$ in the same space $\in S^{1/2} _{1/2} (\rd)$. Let us simply call \emph{Gabor super-frame} $\G(g,\Lambda)$ for $\lrd$ a Gabor frame with the above properties.
The Gabor super-frames are the key for the discretization of the kernel in \eqref{unobis2s}. First, we need  the preliminary result below, which reflects  the algebra property of Gelfand-Shilov spaces.
For $s\geq 1/2$, $\eps>0$, we  define the following weight functions:
\begin{equation}\label{pesieps} w_{s,\eps}(z):=\exp\big({-\eps|z|^{1/s}}\big),\quad z \in\rdd.
\end{equation}

\begin{lemma} \label{algpesi} Let $\Lambda\subset \rdd$  be a lattice of $\rdd$.
Then the sampling $\{w_{s,\eps}(\lambda)\}_{\lambda\in\Lambda}$ of \eqref{pesieps} (defined on $\rdd$) satisfies
\begin{equation}\label{convpesi} (w_{s,\eps}\ast w_{s,\eps})(\lambda):=\sum_{\nu\in\Lambda}w_{s,\eps}(\lambda-\nu)w_{s,\eps}(\nu)\lesssim
w_{s,\,\eps 2^{-1/s}}(\lambda).
\end{equation}
\end{lemma}
\begin{proof} We use the arguments of \cite[Lemma 11.1.1(c)]{grochenig}. For $\lambda \in \Lambda$, we divide the lattice $\Lambda$ into the subsets $N_\lambda=\{\nu\in\Lambda\,: |\lambda-\nu|\leq |\lambda|/2\}$ and
$N^c_\lambda=\{\nu\in\Lambda\,: |\lambda-\nu|>|\lambda|/2\}$. For $\nu\in N_\lambda$, $|\nu|\geq |\lambda|/2$ and $|\nu|^{1/s}\geq (|\lambda|/2)^{1/s}$, so
\begin{equation*}(w_{s,\eps}\ast w_{s,\eps})(\lambda)\leq e^{-(\eps 2^{-\frac1s})|\lambda|^{\frac1s}}\left(\sum_{\nu\in N_\lambda} e^{-\eps|\lambda-\nu|^{\frac1s}}+\sum_{\nu\in N^c_\lambda} e^{-\eps|\nu|^{\frac1s}}\right)
\lesssim e^{-(\eps 2^{-\frac1s})|\lambda|^{\frac1s}}.
\end{equation*}
This concludes the proof.
\end{proof}
\begin{theorem}\label{equivdiscr-cont}  Let $\G(g,\Lambda)$ a Gabor super-frame for $\lrd$. Consider $m\in\mathcal{M}_v(\rdd)$, $s\geq1/2$,
 and a symbol  $\sigma\in\cC^\infty(\rdd)$. Then the following properties are equivalent:
\par
{(i)} There exists $\eps>0$ such that the estimate \eqref{unobis2s} holds.\par
{(ii)} There exists $\eps>0$ such that
\begin{equation}\label{unobis2discr} |\langle \sigma^w \pi(\mu)
g,\pi(\lambda)g\rangle|\lesssim m\left(\frac{\lambda+\mu}2\right) \exp\big({-\eps|\lambda-\mu|^{1/s}}\big),\qquad \forall\,
\lambda,\mu\in\Lambda.
\end{equation}
\end{theorem}
\begin{proof} It remains to show that $(ii)\Rightarrow (i)$. The pattern of \cite[Theorem 3.2]{charly06} can be adapted to this proof by using  a Gabor super-frame $\G(g,\Lambda)$, with a dual window $\gamma\in S^{1/2} _{1/2} (\rd)$.
\par
 Let $\mathcal{Q}$ be a symmetric relatively compact fundamental domain of the lattice $\Lambda\subset\rdd$. Given $w,z\in\rdd$, we can write them uniquely as $w=\lambda+u$, $z=\mu+u'$, for $\lambda,\mu\in\Lambda$ and $u,u'\in \mathcal{Q}$. Using the Gabor reproducing formula for the time-frequency shift $\pi(u)g\in S^{1/2} _{1/2} (\rd)$ we can write
\begin{equation*} \pi(u)g=\sum_{\nu\in\Lambda}\la\pi(u)g,\pi(\nu)\gamma\ra\pi(\nu) g.
\end{equation*}
Inserting the prior expansions in the assumption \eqref{unobis2discr},
\begin{align}
&|\langle \sigma^w \pi(\mu+u')
g,\pi(\lambda+u)g\rangle|\notag\\
&\qquad\qquad\leq \sum_{\nu,\nu'\in\Lambda} |\langle \sigma^w \pi(\mu+\nu')g,\pi(\lambda+\nu)g\rangle|\,|\la \pi(u')g,\pi(\nu')\gamma\ra|\,|\la\pi(u)g,\pi(\nu)\gamma\ra|\notag\\
&\qquad\qquad\lesssim \sum_{\nu,\nu'\in\Lambda} m\left(\frac{\lambda+\mu+\nu+\nu'}2\right) e^{-\eps|\lambda+\nu-\mu-\nu'|^{\frac1s}}|V_\gamma g(\nu'-u')| |V_\gamma g(\nu-u)|\label{mag1}.
\end{align}
Since the window functions $g,\gamma$ are both in $ S^{1/2} _{1/2} (\rd)$,  the  STFT $V_\gamma g$ is in $S^{1/2} _{1/2} (\rdd)$ in view of \eqref{zimmermann1}. Thus there exists $h>0$ such that $ |V_\gamma g(z)|\lesssim e^{-h |z|^{2}}$, for every $z\in\rdd$.
In particular, being $\mathcal{Q}$  relatively compact and $u\in \mathcal{Q}$, $$ |V_\gamma g(\nu-u)|\lesssim e^{-h |\nu-u|^{2}}\leq \sup_{u\in Q}e^{-h |\nu-u|^{2}}\lesssim e^{-h |\nu|^{2}}.$$ The assumption $m\in\mathcal{M}_v(\rdd)$ yields
$$m\left(\frac{\lambda+\mu+\nu+\nu'}2\right)\lesssim m\left(\frac{\lambda+\mu}2\right)v\left(\frac{\nu}2\right)v\left(\frac{\nu'}2\right)
$$
and, for every $0<\tilde{h}<h$,
$$v\left(\frac{\nu}2\right)e^{-h |\nu|^{2}} \lesssim e^{-\tilde{h} |\nu|^{2}},\quad \forall\nu\in\Lambda.
$$
Inserting these estimates in \eqref{mag1},
\begin{align}&|\langle \sigma^w \pi(\mu+u')
g,\pi(\lambda+u)g\rangle|\notag\\
&\qquad\qquad\lesssim m\left(\frac{\lambda+\mu}2\right)\sum_{\nu,\nu'\in\Lambda} e^{-\eps|\lambda+\nu-\mu-\nu'|^{\frac1s}}v\left(\frac{\nu}2\right)e^{-h |\nu|^{2}}v\left(\frac{\nu'}2\right)e^{-h |\nu'|^{2}}\notag\\
&\qquad\qquad\lesssim m\left(\frac{\lambda+\mu}2\right)\sum_{\nu,\nu'\in\Lambda} e^{-\eps|\lambda+\nu-\mu-\nu'|^{\frac1s}}e^{-\tilde{h} |\nu|^{2}}e^{-\tilde{h} |\nu'|^{2}}\notag\\
&\qquad\qquad\leq m\left(\frac{\lambda+\mu}2\right)\sum_{\nu,\nu'\in\Lambda} e^{-b (|\lambda+\nu-\mu-\nu'|^{\frac1s}+ |\nu|^{\frac1s}+|\nu'|^{\frac1s})}\label{convpe}
\end{align}
with $b=\eps$ for $s>1/2$ whereas $b=\min\{\eps,\tilde{h}\}$ for $s=1/2$ (there may be a loss of decay).
We observe that the row \eqref{convpe} can be rewritten as
$$m\left(\frac{\lambda+\mu}2\right)(w_{s,b}\ast w_{s,b}\ast w_{s,b})(\lambda-\mu), $$
where $w_{s,b}(\lambda)=e^{-b|\lambda|^{1/s}}$.  Now we apply Lemma \ref{algpesi} twice and we obtain
\begin{equation}\label{bo}
\eqref{convpe}\lesssim m\left(\frac{\lambda+\mu}2\right) e^{-\tilde{\eps}|\lambda-\mu|^{\frac1s}}
\end{equation}
with $\tilde{\eps}=b 2^{-2/s}$. \par
If $w,z\in\rdd$ and $w=\lambda+u$, $z=\mu+u'$, $\lambda,\mu\in\Lambda$, $u,u'\in\mathcal{Q}$, then $\lambda-\mu=w-z+u'-u$ and $u'-u\in \mathcal{Q}-\mathcal{Q}$, which is a relatively compact set, thus
\begin{equation}\label{stimaexp}e^{-\tilde{\eps}|\lambda-\mu|^{\frac1s}}\lesssim \sup_{u\in \mathcal{Q}-\mathcal{Q}}e^{-\tilde{\eps}|w-z+u|^{\frac1s}}\lesssim e^{-\tilde{\eps}|w-z|^{\frac1s}}.
\end{equation}
Finally, the $v$-moderateness of the weight $m\in\mathcal{M}_v(\rdd)$, together with the fact that $v$ is continuous and the set $\mathcal{Q}+\mathcal{Q}$ is  relatively compact, let us write
\begin{align}\label{stimapesi} m\left(\frac{\lambda+\nu}{2}\right)&=m\left(\frac{w+z}{2}-\frac{u+u'}{2}\right)\lesssim m\left(\frac{w+z}{2}\right)v\left(-\frac{u+u'}{2}\right)\\
&\lesssim m\left(\frac{w+z}2\right)\sup_{u\in \mathcal{Q}+\mathcal{Q}}v\left(\frac{u}2\right)\lesssim m\left(\frac{w+z}2\right)\notag
\end{align}
Combining the estimates \eqref{stimaexp} and \eqref{stimapesi} with \eqref{bo} we obtain \eqref{unobis2s}, with the parameter $\eps=\tilde{\eps}$ which appears in \eqref{bo}.
\end{proof}
\subsection{Sparsity of the Gabor matrix}\label{sec5}
 The operators $\sigma^w$ which satisfy Theorem \ref{CGelfandPseudo}, say with $m=v=1$, enjoy a fundamental sparsity property. Indeed, let $\G(g,\Lambda)$ be a Gabor frame for $\lrd$, with $g\in S^s_s(\rd)$, $s\geq 1/2$.  Then, as we saw,
 \begin{equation}\label{unobis2discrbis} |\langle \sigma^w \pi(\mu)
g,\pi(\lambda)g\rangle|\leq C \exp\big({-\eps|\lambda-\mu|^{1/s}}\big),\qquad \forall\,
\lambda,\mu\in\Lambda,
\end{equation}
with suitable constants $C>0$, $\eps>0$. This gives at once an exponential-type sparsity,
in the sense precised by the
following proposition (cf.\ 
\cite{candes, guo-labate} for the more standard notion of super-polynomial sparsity).
\begin{proposition}
Let the Gabor matrix
$\langle \sigma^w \pi(\mu)
g,\pi(\lambda)g\rangle$ satisfy \eqref{unobis2discrbis}. Then it is sparse in the following sense.
Let $a$ be any column
or row of the matrix, and let
$|a|_n$ be the $n$-largest
entry of the sequence $a$.
Then, $|a|_n$
satisfies
\[
|a|_n\leq C \displaystyle \exp\big({-\epsilon n^{1/(2ds)}}\big),\quad n\in\bN
\]
\end{proposition}
for some constants $C>0,\epsilon>0$.
\begin{proof}
By a discrete analog of Proposition \ref{equi} it suffices to prove that
\[
n^\alpha |a|_n\leq C^{\alpha+1}(\alpha!)^{2ds},\quad \alpha\in\mathbb{N}.
\]
On the other hand we have
\[
n^{\frac1p}\cdot|a|_n\leq
\|a\|_{\ell^p},
\]
for every $0<p\leq\infty$. Hence by \eqref{unobis2discrbis} and setting $p=1/\alpha$ we obtain
\[
n^\alpha |a|_n\leq \Big(\sum_{\lambda\in\Lambda} e^{-\epsilon p|\lambda-\mu|^{\frac1s}}\Big)^{\frac1p}
=
 \Big(\sum_{\lambda\in\Lambda} e^{-\epsilon p|\lambda|^{\frac1s}}\Big)^{\frac1p}.
\]
Let $\mathcal{Q}$ be a fundamental domain of the lattice $\Lambda$. Then if $x\in\lambda+\mathcal{Q}$, $\lambda\in\Lambda$, we have $|x|\leq |\lambda|+C_0$, therefore $|x|^{1/s}\leq C_1(|\lambda|^{1/s}+1)$. Hence
\begin{align*}
\sum_{\lambda\in\Lambda} e^{-\epsilon p|\lambda|^{\frac1s}}&\leq C_2 \int_{\rdd} e^{-\epsilon p|x|^{\frac1s}}\, dx=\int_{\mathbb{S}^{2d-1}}d\sigma\int_0^{+\infty} e^{-\eps p \rho^{\frac 1s}}\rho^{2d-1}d \rho\\
&= \frac {C_3s}{(\eps p)^{2ds}}\int_0^{+\infty} e^{-t} t^{2d s-1} dt= \frac{C_3s \Gamma(2ds)}{(\eps p)^{2ds}} =\frac{C_4}{p^{2ds}}
\end{align*}
Finally, by Stirling's formula,
\[
n^\alpha |a|_n\leq \frac{C_4^{1/p} }{p^{\frac{2ds}{p}}}\leq C_5^{\alpha+1}(\alpha!)^{2ds}.
\]
%\[
%n^\alpha |a|_n\leq C_2^{1/p} \Big(\int_{\rdd} e^{-\epsilon p|x|^{1/s}}\, dx\Big)^{1/p}=C_3^{1/p}(1/p)^{2ds/p}\leq C_4^{\alpha+1}(\alpha!)^{2ds},
%\]
%by Stirling's formula.
\end{proof}
\subsection{Boundedness of pseudodifferential operators on modulation and Gelfand-Shilov spaces}
We list here some continuity results on modulation and Gelfand-Shilov spaces which follow easily from our basic Theorem \ref{CGelfandPseudo}.
\begin{proposition}\label{modsp}
Let $s\geq 1/2$ and consider a symbol $\sigma\in \cC^\infty(\rdd)$ satisfying the estimates
\begin{equation}\label{stimas}
|\partial^\a \sigma(z)|\lesssim C^{|\a|}(\a!)^s,
\end{equation}
for some $C>0$.
Let $m\in \mathcal{M}_v(\rdd)$ and, if $s\geq1$, assume the weight $v$ satisfies
 \[
 v(z)\lesssim \exp\big({\epsilon |z|^{1/s}}\big),\quad z\in\rdd
 \]
for every $\epsilon>0$.
 Then the Weyl operator
$\sigma^w$ extends to a bounded operator on $\mathcal{M}_m^{p,q}(\R^d)$.
\end{proposition}
\begin{proof}
Let $g \in \Sigma_1^1(\rd)$ with $\| g \|_{L^2}=1$.
From the inversion formula \eqref{invformula},
\[
V_g (\sigma^w f)(u)=\int_{\R^{2d}}\langle \sigma^w \pi(z) g,
\pi(u)g\rangle \, V_g f(z) \, dz.
\]
The desired result thus follows if we can prove that the map $M(\sigma)$
defined by
\[
M(\sigma) G(u)=\int_{\R^{2d}}\langle  \sigma^w \pi(z) g,
\pi(u) g \rangle \, G(z) \, dz
\]
is continuous from $L^{p,q}_m(\rdd)$ into $L^{p,q}_m(\rdd)$. The characterization of Theorem \ref{CGelfandPseudo} assures the existence of an $\eps>0$, such that
$$|V_g (\sigma^w f)|=|M(\sigma) V_g (f)|\lesssim (w_{s,\eps}\ast|V_g f|)(u),\quad u\in \rdd
$$
where the weight function $w_{s,\eps}$ on $\rdd$ is defined in \eqref{pesieps}. The desired conclusion then follows from the relation $L^{p,q}_m\ast L^1_v\hookrightarrow L^{p,q}_m$, for $w_{s,\epsilon}\in L^1_v$ by the growth assumption on $v$ (if $s<1$, \eqref{weight} suffices).

%Since  $f \in \Sigma_1^1(\rd)$ if and only if $V_g f \in \Sigma_1^1(\rdd)$, cf.\  \cite{GZ}, this implies $|V_g f(z)|\lesssim e^{-h|z|}$, for every $z\in\rdd$ and every $h>0$. Thus, by the submultiplicativity of $w_{1,h}$, $$|M(\sigma) G(u)|\lesssim (w_{s,\eps}\ast w_{1,h})(u)\lesssim w_{1,h}(u),\quad \forall u\in\rdd. $$ The assumption $m\in\mathcal{M}_v$ implies there exists a $k>0$ such that $|m(z)|\lesssim e^{k|z|}$. Thus, choosing  $h>k$, we obtain $w_{1,h}\in L^{p,q}_m$ and the claim is proved.

\end{proof}

\begin{proposition}\label{gsooo}
Let $s\geq 1/2$, and consider a symbol $\sigma\in \cC^\infty(\rdd)$ that satisfies
\eqref{stimas}. Then the Weyl operator $\sigma^w$ is bounded on $S^s_s(\rd)$.
\end{proposition}
\begin{proof} Fix a window function $g\in S^{1/2}_{1/2}(\rd)$. For $f\in S^s_s(\rd)$ we have $V_g f\in S^s_s(\rdd)$ in view of \eqref{zimmermann1}. Hence there exists a constant $h>0$ such that $|V_g f|(z)\lesssim e^{-h|z|^{1/s}}$, for every $z\in\rdd$. Thus, taking $\tilde{\eps}=\min\{\eps,h\}$ and using a continuous version of Lemma \ref{algpesi}, for every $u\in\rdd$,
\[
|V_g (\sigma^w f)(u)|\lesssim (w_{s,\eps}\ast w_{s,h})(u)\lesssim  (w_{s,\tilde{\eps}}\ast w_{s,\tilde{\eps}})(u)\lesssim 
w_{s,\,\tilde{\eps} 2^{-1/s}}(u).
\]
Using \eqref{zimmermann2} we obtain the claim.
\end{proof}

Boundedness results for pseudodifferential operators on Gelfand-Shilov spaces are also contained in \cite{ToftGS}.
\section{Applications to evolution equations}
In this section we apply the above almost diagonalization result to the propagators for certain constant coefficient evolution equations. \par
We consider an operator
\begin{equation}\label{operin}
P(\partial_t,D_x)=\partial_t^m +\sum_{k=1}^{m}a_k(D_x)\partial_t^{m-k},\quad t\in\R,\ x\in\R^d,
\end{equation}
%Elena%%%%%%%%%%%%%%%%%%%%
where $a_{k}(\xi)$, $1\leq k\leq m$, are polynomials.
%%%%%%%%%%%%%%%%%%%%%%%%%
 They may be non-homogeneous, and their degree may be arbitrary.  Here we set $D_{x_j}=\frac{1}{2\pi i}\partial_{x_j}$, $j=1,\ldots,d$.\par
We are interested in the forward Cauchy problem
\[
\begin{cases}
P(\partial_t,D_x)u=0,\quad (t,x)\in\R_+\times\rd\\
\partial_t^k u(0,x)=u_k(x),\quad 0\leq k\leq m-1,
\end{cases}
\]
%Elena%%%%%%%%%%%%%%
where $u_k\in\cS(\rd)$, $0\leq k\leq m-1$.
%%%%%%%%%%%%%%%%%%%%
Hence we suppose that the {\it forward Hadamard-Petrowsky condition} is satisfied: {\it There exists a constant $C>0$ such that}
\begin{equation}\label{hp}
(\tau,\xi)\in\mathbb{C}\times\rd,\quad P(i\tau,\xi)=0 \Longrightarrow {\rm Im}\, \tau\geq -C.
\end{equation}
This is a sufficient and necessary condition for the above Cauchy problem with Schwartz data to be well posed \cite[Section 3.10]{rauch}.
By taking the Fourier transform with respect to $x$ and using classical results for the fundamental solution to ordinary differential operators \cite[pp. 126-127]{schwartz}, one sees that the solution is then given by
\[
u(t,x)=\sum_{k=0}^{m-1} \partial_t^k E(t,\cdot)\ast \Big(u_{m-1-k}+\sum_{j=1}^{m-k-1} a_j(D_x) u_{m-k-1-j}\Big).
\]
Here $E(t,x)=\Fur^{-1}_{\xi\to x} \sigma(t,\xi)$, where $\sigma(t,\xi)$ is the unique solution to
\[
\Big(\partial^m_t+\sum_{k=1}^m a_k(\xi) \partial_t^{m-k}\Big)\sigma(t,\xi)=\delta(t)
\]
supported in $[0,+\infty)\times\rd$. The distribution $E(t,x)$ is therefore the  fundamental solution of $P$ supported in $[0,+\infty)\times\rd$. \par
The study of the Cauchy problem is therefore reduced to that of the Fourier multiplier
\begin{equation}\label{hpm0}
\sigma^w(t,D_x)=\sigma(t,D_x) f =\Fur^{-1} \sigma(t,\cdot)\Fur f=E(t,\cdot)\ast f.
\end{equation}
(For Fourier multipliers the Weyl and Kohn-Nirenberg quantizations give the same operator).
\begin{example}\rm 
Here are some classical operators and the symbols $\sigma(t,\xi)$ of the corresponding propagators for $t\geq0$ ($\sigma(t,\xi)=0$ for $t<0$): \par
Wave operator $\partial^2_t-\Delta$; $\sigma(t,\xi)=\frac{\sin(2\pi|\xi|t)}{2\pi|\xi|}$.\par
Klein-Gordon operator $\partial^2_t-\Delta+m^2$, with $m>0$; $\sigma(t,\xi)=\frac{\sin(t\sqrt{4\pi^2|\xi|^2+m^2})}{\sqrt{4\pi^2|\xi|^2+m^2}}$.\par
Heat operator $\partial_t-\Delta$; $\sigma(t,\xi)=e^{-4\pi^2|\xi|^2 t}$.
\end{example}
We now introduce an assumption under which the multiplier $\sigma(t,D_x)$ falls in the class of pseudodifferential operators considered above. \par\medskip
{\it Assume that there are constants $C>0$, $\nu\geq1$ such that}
\begin{equation}\label{hpm}
(\tau,\zeta)\in\mathbb{C}\times\mathbb{C}^d,\quad P(i\tau,\zeta)=0 \Longrightarrow {\rm Im}\, \tau\geq -C(1+|{\rm Im}\,\zeta|)^\nu.
\end{equation}
It is clear that this condition is stronger than the forward Hadamard-Petrowsky condition.
\begin{theorem}\label{E4}
%Elena
Assume $P$ satisfies \eqref{hpm} for some $C>0$, $\nu\geq1$. Then the symbol $\sigma(t,\xi)$ of the corresponding propagator $\sigma(t,D_x)$ in \eqref{hpm0} satisfies the following estimates:
\begin{equation}\label{E3}
|\partial^\alpha_{\xi} \sigma(t,\xi)|\leq C^{(t+1)|\alpha|+t} (\alpha!)^s,\quad \xi\in\rd,\ t\geq0,\quad \alpha\in\bN^d,
\end{equation}
with $s=1-1/\nu$, for a new constant $C>0$.
%%%%%%%%%%%%%%%%%%%%%%%
\end{theorem}
\begin{proof}
It is well known (see e.g.\  \cite[Proposition 1.3.2, Lemma 1.3.3]{cattabriga}, \cite[Section 3.10]{rauch}) that $\sigma(t,\xi)$ extends to an entire analytic function in the second variable, and that the estimates \eqref{hpm}  imply the bound
\begin{equation}\label{hpm2}
| \sigma(t,\zeta)|\leq e^{Ct(1+|{\rm Im}\, \zeta|)^\nu},\quad \zeta\in\mathbb{C}^d.
\end{equation}
%Elena
Now, given $\xi\in\rd$, we consider  the polydisk $B(\xi,R)=\prod_{j=1}^d B_j(\xi_j,R)=\{\zeta\in\mathbb{C}^d:\, |\zeta_j-\xi_j|\leq R, 1\leq j\leq d\}$, with $R=(1+|\alpha|)^{1/\nu}$, $\alpha\in\bN^d$. Observe that \eqref{hpm2} implies
\begin{equation}\label{E1}
\sup_{\zeta\in B(\xi,R)}|\sigma(t,\zeta)|\leq e^{Ct(1+\sqrt{d}R)^\nu}.
\end{equation}
The Cauchy's Generalized Integral Formula
$$\partial^\alpha_{\xi} \sigma(t,\xi)=\frac{\a!}{(2\pi i)^d}\int\cdots\int_{\partial B_1(\xi_1,R)\times\cdots\times\partial B_d(\xi_d,R)}\frac{\sigma(t,\zeta_1,\dots,\zeta_d)}{(\zeta_1-\xi_1)^{\a_1+1}\cdots(\zeta_d-\xi_d)^{\a_d+1}}d\zeta_1\cdots d\zeta_d
$$
and the estimate \eqref{E1} yield
\begin{equation}\label{E2}
|\partial^\alpha_{\xi} \sigma(t,\xi)|\leq \frac{\alpha! e^{Ct(1+\sqrt{d}R)^\nu}}{R^{|\alpha|}}\leq C_1^{t(|\alpha|+1)}\frac{\alpha!}{{(1+|\alpha|)}^{|\alpha|/\nu}}.
\end{equation}
Using Stirling formula and $1/\nu=1-s$, we have
$$\frac{1}{(1+|\alpha|)^{|\alpha|/\nu}}\leq \frac{C_2^{|\a|}}{(\a!)^{1-s}},
$$
which combined with  \eqref{E2} provides the desired majorization \eqref{E3}.
%\[
%\sup_{\zeta\in B(\xi,R)}|\sigma(t,\zeta)|\leq e^{Ct(1+\sqrt{d}R)^s}\] we get
%\[
%|\partial^\alpha_{\xi} \sigma(t,\xi)|\leq \frac{\alpha! e^{Ct(1+\sqrt{d}R)^s}}{R^{|\alpha|}}\leq C_1^{t(|\alpha|+1)}\frac{\alpha!}{{(1+|\alpha|)}^{|\alpha|/s}}\leq C_2^{(t+1)|\alpha|+t}(\alpha!)^\mu.
%\]
\end{proof}
%Elena%%%%%%%%%%%%%%%%%%%%%%

Observe that the assumption $\nu\geq 1$ in the above theorem implies $0\leq s<1$. \par
Combining Theorem \ref{E4} and  Theorem \ref{CGelfandPseudo} we obtain at once our main application.
\begin{theorem}\label{teo5.2}
Assume $P$ satisfies \eqref{hpm} for some $C>0$, $\nu\geq1$, and set  $r=\min\{2,\nu/(\nu-1)\}$. If $g\in S^{1/r}_{1/r}(\rd)$ then  $\sigma(t,D_x)$  in \eqref{hpm0} satisfies
\begin{equation}\label{hpm3} |\langle \sigma(t,D_x) \pi(z)
g,\pi(w)g\rangle|\leq C \exp\big({-\eps |w-z|^{r}}\big),\qquad \forall\,
z,w\in\rdd,
\end{equation}
for some $\epsilon>0$ and for a new constant $C>0$. The constants $\epsilon$ and $C$ are uniform when $t$ lies in bounded subsets of $[0,+\infty)$.
\end{theorem}
Notice that in \eqref{hpm3} we have $r>1$, so that we always obtain super-exponential decay. \par
%%%%%%%%%%%%%%%%%%%%%%%%%%%%%
%\begin{theorem}\label{teo5.2}
%Assume \eqref{hpm} for some $C>0$, $s\geq1$, and let $r=\min\{2,s/(s-1)\}$, $g\in S^r_r(\rd)$ (it is understood that $s/(s-1)=0$ if $s=1$). The multiplier $\sigma(t,D_x)$ in \eqref{hpm0} satisfies
%\begin{equation}\label{hpm3} |\langle \sigma(t,D_x) \pi(z)
%g,\pi(w)g\rangle|\leq C e^{-\eps |w-z|^{r}},\qquad \forall\,
%z,w\in\rdd,
%\end{equation}
%for some $\epsilon>0$ and for a new constant $C>0$. The constants $\epsilon$ and $C$ are uniform when $t$ lies in bounded subsets of $[0,+\infty)$.
%\end{theorem}
%Notice that in the above theorem we always obtain super-exponential decay, i.e.\ $r>1$ in \eqref{hpm3}. \par
We now show that, if $P(\partial_t,D_x)$ is any hyperbolic operator, then \eqref{hpm} is satisfied with $\nu=1$, and hence the above theorem applies with Gaussian decay ($r=2$ in \eqref{hpm3}), for windows $g\in S^{1/2}_{1/2}(\rd)$. \par
We recall that the operator $P(\partial_t,D_x)$ is called hyperbolic with respect to $t$ if the direction $N=(1,0,\ldots,0)\in\R\times\rd$ is non characteristic for $P$ (i.e.\ its principal symbol -- the higher order homogeneous part in the symbol -- does not vanish at $N$), and $P$ satisfies the forward Hadamard-Petrowsky condition \eqref{hp}. It follows then that the operators $a_k(D_x)$ in \eqref{operin} must have degree $\leq k$ and $P$ has order $m$. \par
The wave and Klein-Gordon operators are of course the most important examples of hyperbolic operators. We emphasize, however, that $P$ is {\it not} required to be strictly hyperbolic, namely the roots of the principal symbol are allowed to coincide. For example, the operator $P=\partial_t^2-\sum_{j,k=1}^d a_{j,k}\partial_{x_j}\partial_{x_k}$, is hyperbolic if the matrix $a_{j,k}$ is real, symmetric and positive {\it semi}-definite.
\begin{proposition}\label{pro5.4}
Assume $P(\partial_t,D_x)$ is hyperbolic with respect to $t$. Then the condition \eqref{hpm} is satisfied with $\nu=1$ for some $C>0$, and hence
\begin{equation} |\langle \sigma(t,D_x) \pi(z)
g,\pi(w)g\rangle|\leq C \exp\big({-\eps |w-z|^{2}}\big),\qquad \forall\,
z,w\in\rdd,
\end{equation}
if $g\in S^{1/2}_{1/2}(\rd)$, for some $\epsilon>0$ and for a new constant $C>0$.
\end{proposition}
\begin{proof}
Denote by $P_m$ the principal symbol of $P$ and $\Gamma(P,N)$ the component of $N$ in $\{\theta\in\R\times \rd: P_m(\theta)\not=0\}$. It follows from the hyperbolicity assumption (see e.g.\  \cite[Proposition 12.4.4]{hormander}) that  the symbol $Q(\tau,\xi)=P(i\tau,\xi)$ of $P(\partial_t,D_x)$ satisfies
\[
Q(\Xi+\tau N+\sigma\theta)\not=0\quad \ if\ \Xi\in\R^{d+1},\ {\rm Im}\,\tau<-C,\ {\rm Im}\,\sigma\leq0,\ \theta\in\Gamma(P,N),
\]
for the same constant $C$ which appears in \eqref{hp}.
We then deduce that
\[
(\tau,\zeta)\in\bC\times\bC^d,\ \ Q(\tau,\zeta)=P(i\tau,\zeta)=0 \Longrightarrow ({\rm Im}\, \tau,{\rm Im}\,\zeta)\not\in -C N-\Gamma(P,N),
\]
Indeed, if $(\tau,\zeta)=\Xi+i(-C N-\theta)$, with $\theta\in\Gamma(P,N)$, $\Xi\in\R^{d+1}$, then $\theta-\epsilon N\in \Gamma(P,N)$ is $\epsilon$ is small enough because $\Gamma(P,N)$ is open, and then
\[
P(\tau,\zeta)=P(\Xi-i(C+\epsilon) N-i(\theta-\epsilon N))\not=0.
\]
Hence $-(C+{\rm Im}\, \tau,{\rm Im}\,\zeta)\not\in \Gamma(P,N)$, and since the cone $\Gamma(P,N)$ is open, this implies
\[
\frac{-(C+{\rm Im}\, \tau,{\rm Im}\,\zeta)\cdot N}{|(C+{\rm Im}\, \tau,{\rm Im}\,\zeta)|}=\frac{-(C+{\rm Im}\, \tau)}{|(C+{\rm Im}\, \tau,{\rm Im}\,\zeta)|}\leq C_1,
\]
for some constant $C_1<1$, which gives
\[
{\rm Im}\, \tau\geq -C-\frac{C_1}{1-C_1}|{\rm Im}\,\zeta|.
\]
\end{proof}
\begin{example}\label{esempio}\rm
Consider the wave operator $P=\partial^2_t-\Delta$ in $\R\times\R^d$, hence $\sigma(t,\xi)=\frac{\sin(2\pi|\xi|t)}{2\pi|\xi|}$. Using the expression for its forward fundamental solution we can estimate directly the matrix decay. Consider, for simplicity, the case of dimension $d\leq 3$. We take  $g(x)=2^{d/4}e^{-\pi|x|^2}$ as  window function, which is allowed because $g\in S^{1/2}_{1/2}(\rd)$, and moreover $\|g\|_{L^2}=1$ (Gaussian functions minimize the Heisenberg uncertainty so that they are, generally speaking, a natural choice for wave-packet decompositions). We claim that
\[
|\langle \sigma(t,D_x) M_{\xi} T_{x} g, M_{\xi'} T_{x'} g\rangle|\leq t e^{-\frac{\pi}{2}[|\xi'-\xi|^2+(|x'-x|-t)_+^2]},\quad x,x',\xi,\xi'\in\rd,\quad d\leq3,
\]
where $(\cdot)_+$ denote positive part. This agrees with Proposition \ref{pro5.4}, with the constants made explicit.
 \par
In fact, in dimension $d\leq3$ we know that
\[
\sigma(t,D_x)f(x)=\int f(x-y)\, d\mu_t(y)
\]
where $\mu_t$ is a  positive Borel measure, supported in the ball $B(0,t)$, with total mass $=t$ (see e.g.\  \cite[Chapter 4]{rauch}). To be precise, we have $d\mu_t(y)= (1/2)\chi_{[-t,t]}dy$ in dimension $1$, $d\mu_t(y)=(2\pi)^{-1}(t^2-|y|^2)_+^{-1/2}dy$ in dimension $2$, and $d\mu_t= (4\pi t)^{-1} d\sigma_{\partial B(0,t)}$ in dimension $3$ (surface measure).
\par
Hence, using $T_y M_\xi=e^{-2\pi i y\xi} M_\xi T_y$,
\[
\langle \sigma(t,D_x) M_{\xi} T_x g, M_{\xi'} T_{x'} g\rangle= \int e^{-2\pi i y\xi}\langle M_{\xi} T_{x+y} g, M_{\xi'} T_{x'} g\rangle \, d\mu_t(y).
\]
An explicit computation shows that
\[
|\langle \pi(z) g,\pi(w) g\rangle|= e^{-\frac{\pi}{2}|w-z|^2},\quad z,w\in\rdd,
\]
so that
\begin{align*}
|\langle \sigma(t,D_x) M_{\xi} T_x g, M_{\xi'} T_{x'} g\rangle|&\leq \int e^{-\frac{\pi}{2}[|\xi'-\xi|^2+|x'-x-y|^2]} \, d\mu_t(y)\\
&\leq t e^{-\frac{\pi}{2}[|\xi'-\xi|^2+(|x'-x|-t)_+^2]},
\end{align*}
which gives the claim.\par
Consider now the Gabor frame $\mathcal{G}(g,\Lambda)$, with $g(x)=2^{d/4}e^{-\pi|x|^2}$, $\Lambda=\bZ^d\times (1/2)\mathbb{Z}^d$ (\cite[Theorem 7.5.3]{grochenig}), and the Gabor matrix
\[
T_{m',n',m,n}=\langle \sigma(t,D_x) M_{n} T_m g, M_{n'} T_{m'} g \rangle,\quad (m,n),\ (m',n')\in\Lambda.
\]
We therefore have
\[
|T_{m',n',m,n}|\leq \tilde{T}_{m',n',m,n}:=  t e^{-\frac{\pi}{2}[|n'-n|^2+(|m'-m|-t)_+^2]},\quad (m,n),\ (m',n')\in\Lambda,\quad d\leq3.
\]
%Figure \ref{figura15} shows the magnitude of the entries, rearranged in decreasing order, of a generic column, e.g.\  $\tilde{T}_{m',n',0,0}$ (obtained for $m=n=0$), at time $t=0.75$, in dimension $d=2$. In fact, the same figure applies to all columns, for $\tilde{T}_{m',n',m,n}=\tilde{T}_{m'-m,n'-n,0,0}$. This figure should be compared with \cite[Figure 15]{cddy}, where a similar investigation was carried out for the curvelet matrix of the wave propagator on the unit square ($d=2$) with periodic boundary conditions. It turns out that the Gabor decay is even better, in spite of the fact that we consider here the wave operator in the whole $\R^2$.
 \end{example}

We now present a class of examples of operators which satisfy Theorem \ref{teo5.2} and are not hyperbolic.
\begin{example}\label{esempio-calore}\rm
Consider the operator
\begin{equation}\label{exa14}
P(\partial_t,D_x)=\partial_t+(-\Delta)^k,
\end{equation}
with $k\geq 1$ integer. When $k=1$ we have the heat operator. Its symbol is the polynomial
\[
P(i\tau,\zeta)=i\tau+(4\pi^2 \zeta^2)^k,
\]
We claim that it satisfies \eqref{hpm} with $\nu=2k$. \par
If $\tau\in\bC,\ \xi,\eta\in\R^d$ and $P(i\tau,\xi+i\eta)=0$ then
\[
{\rm Im}\, \tau=(2\pi)^{2k}{\rm Re} (|\xi|^2+2i\xi\cdot\eta-|\eta|^2)^k=(2\pi)^{2k} |\xi|^{2k}+Q(\xi,\eta),
\]
where $Q(\xi,\eta)$ is a homogeneous polynomial of degree ${2k}$, with $Q(\xi,0)=0$. Hence, for every $\epsilon>0$ and some constant $C_\epsilon>0$,
\[
|Q(\xi,\eta)|\lesssim\sum_{j+l=2k, j\geq0, l\geq1} |\xi|^{j}|\eta|^l\leq \epsilon |\xi|^{2k}+C_\epsilon |\eta|^{2k},
\]
 where we applied the inequality $ab\leq \frac{(\epsilon a)^p}{p}+\frac{(b/\epsilon)^q}{q}$ (with $a=|\xi|^j$, $b=|\eta|^l$, $p=2k/j$ and $q=2k/l$) to the terms of the sum with $j\geq1$. Taking $\epsilon=(2\pi)^{2k}/2$ we get
\[
{\rm Im}\, \tau\geq \frac{(2\pi)^{2k}}{2}|\xi|^{2k}-C|\eta|^{2k}\geq -C|\eta|^{2k},
\]
which proves the claim.\par
Hence, for the operator $P$ in \eqref{exa14}, Theorem \ref{teo5.2} applies with $\nu=2k$ and $r=2k/(2k-1)$.
\end{example}
\begin{remark}\label{osservazione-variabile}\rm
Let us add a few words about the Gabor analysis of hyperbolic problems with variable coefficients. In this regards, the example of the transport equation with variable coefficients, standard test for numerical schemes, is somewhat discouraging. In fact, for a fixed time $t>0$, the operator $T$ solving the Cauchy problem turns out to be a change of variables, and Gabor frames are not suited to follow the corresponding non-linear propagation of singularities. Heuristically this can be seen as follows. Decomposition of a function in terms of Gabor atoms $\pi(\lambda)g$ corresponds, geometrically, to a uniform partition of the time-frequency space (or phase space, in PDE's terminology) into boxes, each atom occupying a box, loosely speaking (\cite[pag 211]{grochenig}). The correspondence principle in Quantum Mechanics suggests one to follow the PDE's evolution in terms of its classical analog, namely the Hamiltonian flow in phase space. The propagator has therefore a sparse Gabor matrix if its Hamiltonian flow moves the above mentioned boxes,  but introduces {\it only a controlled number of overlaps}. However, this is not the case for changes of variables, as one sees easily by direct inspection. More rigorously, it follows from the results in \cite{cnr-flp,cnr-global} that H\"ormander-type Fourier integral operators, in particular the changes of variables, do no have a sparse Gabor representation.\par
However, Theorems \ref{CGelfandPseudo} and \ref{equivdiscr-cont} consent also applications to certain equations with variable coefficients. As elementary example consider the transport equation
\[
\begin{cases}
\partial_t u-i\left(\sum_{j=1}^d a_j D_{x_j}+\sum_{j=1}^d b_j x_j\right)u=0\\
u(0,x)=u_0(x),
\end{cases}
\]
with $a_j,b_j\in\R$. The fundamental solution of the problem has symbol
\[
\sigma(t,x,\xi)= \exp\Big[ it\sum_{j=1}^d (a_j \xi_j+b_j x_j)\Big]
\]
satisfying \eqref{simbsmooth} with $m=1$, $s=0$, and the preceding arguments apply. More generally, it seems possible to treat similarly the pseudodifferential problems of the form \[
\partial_t u -i a^w(x,D)  u=0,\quad u(0,x)=u_0(x),
\]
where $a(x,\xi)$ is real-valued and $\partial_{x,\xi} a(x,\xi)$, as well as higher order derivatives, are bounded in $\rdd$, satisfying suitable estimates.
\end{remark}
%\end{proof}

\section*{Acknowledgements}
The authors would like to thank Prof. Karlheinz Gr\"ochenig for his inspiring comments about Gabor frames of Hermite functions.

%%%%%%%%%%%%%%%%%%%%%%%%%%%%%%%%%%%%%%%%%%%%

\end{document}